\documentclass[reqno,12pt,letterpaper]{amsart}
\usepackage{amsmath,amssymb,amsthm,graphicx,mathrsfs,url}
\usepackage[usenames,dvipsnames]{color}
\usepackage[colorlinks=true,linkcolor=Red,citecolor=Green]{hyperref}

\def\?[#1]{\textbf{[#1]}\marginpar{\Large{\textbf{??}}}}
\renewcommand{\Re}{\mathop{\rm Re}\nolimits}
\renewcommand{\Im}{\mathop{\rm Im}\nolimits}

\DeclareMathOperator{\Op}{Op}
\DeclareMathOperator{\Id}{Id}
\DeclareMathOperator{\ad}{ad}
\DeclareMathOperator{\tr}{tr}
\DeclareMathOperator{\loc}{loc}
\DeclareMathOperator{\vol}{vol}
\DeclareMathOperator{\comp}{comp}
\DeclareMathOperator{\rank}{rank}
\DeclareMathOperator{\Ai}{Ai}
\DeclareMathOperator{\Image}{Image}
\DeclareMathOperator{\diag}{diag}
\DeclareMathOperator{\supp}{supp}
\DeclareMathOperator{\esssupp}{ess supp}
\DeclareMathOperator{\Res}{Res}
\DeclareMathOperator{\WF}{WF}

\newtheorem{prop}{Proposition}
\newtheorem{thm}[prop]{Theorem}

\newtheorem{lem}[prop]{Lemma}

\numberwithin{equation}{section}
\numberwithin{prop}{section}

\begin{document}
\title[Resonances of Convex Obstacles]{Scattering Resonances of Convex Obstacles for general boundary conditions}
\author{Long Jin}
\email{jinlong@math.berkeley.edu}
\address{Department of Mathematics, Evans Hall, University of California,
Berkeley, CA 94720, USA}
\date{}
\maketitle

\begin{abstract}
We study the distribution of resonances for smooth strictly convex obstacles under general boundary conditions. We show that under a pinched curvature condition for the boundary of the obstacle, the resonances are separated into cubic bands and the distribution in each bands satisfies Weyl's law.
\end{abstract}

\section{Introduction and Statement of the Results}
\label{sec:intro}

In this paper, we study the distribution of resonances for convex strictly obstacles $\mathcal{O}$ under general boundary conditions including Neumann and general smooth Robin boundary conditions $\partial_\nu u+\eta u=0$, $\eta\in C^\infty(\partial\mathcal{O})$. The goal of this paper is to prove that if the boundary of the obstacle satisfies a pinched curvature condition, then the resonances that are close to the real axis are separated in several bands. We also give the asymptotics formula for the counting functions of resonances in each bands.

Let $\mathcal{O}\subset\mathbb{R}^n$ be a strictly convex obstacle with smooth boundary. More precisely, let $Q$ be the second fundamental form of $\partial\mathcal{O}$ and $S\partial\mathcal{O}$ be the sphere bundle of $\partial\mathcal{O}$, then $\min_{S\partial\mathcal{O}}Q>0$. We shall write

\begin{equation}
\kappa=2^{-1/3}\cos(\pi/6)\min_{S\partial\mathcal{O}}Q^{2/3},\;\;\;
K=2^{-1/3}\cos(\pi/6)\max_{S\partial\mathcal{O}}Q^{2/3}.
\end{equation}

Let $P=-\Delta_{\mathbb{R}^n\setminus\mathcal{O}}$ be the Laplacian operator on the exterior domain $\mathbb{R}^n\setminus\mathcal{O}$ associated with the Neumann/Robin boundary condition which will be defined precisely later in section \ref{sec:prelim}, then the resolvent $R(\lambda)=(-\Delta-\lambda^2)^{-1}$ which is analytic for $\Im \lambda>0$ has a meromorphic continuation to the whole complex plane $\mathbb{C}$ (when $n$ is odd) or the logarithmic covering of $\mathbb{C}\setminus\{0\}$ (when $n$ is even). The poles of $R(\lambda)$ are called resonances or scattering poles.

In \cite{J}, we proved that there are no resonances in the region
\begin{equation}
C\leqslant\Re\lambda,\;\;\; 0\leqslant-\Im\lambda\leqslant\kappa\zeta_1'(\Re\lambda)^{1/3}-C
\end{equation}
where $\zeta_1'$ is the negative of the first zero of the derivative of the Airy function $\Ai'$ and $C$ is some constant. The main result in this paper is to obtain alternating cubic bands with and without resonances. More precisely, let $0<\zeta_1'<\zeta_2'<\cdots$ be the negative of the zeroes of $\Ai'$, then we have the following theorem.

\begin{thm}
\label{thm:main1}
Suppose we have the following pinched curvature condition
\begin{equation}
\frac{\max_{S\partial\mathcal{O}}Q}{\min_{S\partial\mathcal{O}}Q}
<\left(\frac{\zeta'_{j_0+1}}{\zeta'_{j_0}}\right)^{3/2}
\end{equation}
for some $j_0\geqslant1$. Then there exists a constant $C>0$ such that for all $0\leqslant j\leqslant j_0$, there are no resonances in the regions
\begin{equation}
C\leqslant\Re\lambda,\;\;\; K\zeta_j'(\Re\lambda)^{1/3}+C\leqslant-\Im\lambda
\leqslant\kappa\zeta_{j+1}'(\Re\lambda)^{1/3}-C.
\end{equation}
\end{thm}

The Dirichlet case was already established in Sj\"{o}strand-Zworski \cite{SZ6} where a Weyl law for resonances in each band with a rough error term was also given. Their argument can also be directly adapted to our situation to give the following theorem.

\begin{thm}
\label{thm:weyl}
Under the assumption in theorem \ref{thm:main1}, for some $C>0$ and all $1\leqslant j\leqslant j_0$,
\begin{equation}
\label{weyl}
\begin{split}
\sum\{M_\mathcal{O}(\lambda):|\lambda|\leqslant r,&\;
\kappa\zeta_j'(\Re\lambda)^{1/3}-C<-\Im\lambda
<K\zeta_j'(\Re\lambda)^{1/3}+C\}\\
&=(1+o(1))(2\pi)^{1-n}\vol(B^{n-1}(0,1))\vol(\partial\mathcal{O})r^{n-1},
\end{split}
\end{equation}
where $B^{n-1}(0,1)$ is the unit ball in $\mathbb{R}^n$.
\end{thm}

We should point out that for spherical obstacles (for which $\kappa=K$ and the pinched curvature assumption in Theorem \ref{thm:main1} holds trivially for all $j_0$), the resonances can be described using Hankel functions in the case of Dirichlet, Neumann and constant Robin boundary conditions. Each band,
\begin{equation*}
\kappa\zeta_j'(\Re\lambda)^{1/3}-C<-\Im\lambda
<K\zeta_j'(\Re\lambda)^{1/3}+C,
\end{equation*}
between the resonance-free bands actually reduces to a curve which is asymptotically cubic. Moreover, there is a better error in Weyl's law, $O(r^{n-2})$ instead of $o(r^{n-1})$, in this situation. See Stefanov \cite{St} for a detailed discussion of scattering resonances for the sphere. The results of Sj\"{o}strand-Zworski \cite{SZ6} and of this paper show that more bands are separated from each other when the obstacle is closer to a ball (in the sense that the curvatures of the boundary are closer to a constant.)

The problem of the distribution of resonances for convex obstacles has been extensively studied in the literature. For the spherical case, it dates back to Watson's work on the scattering of electromagnetic wave by the earth \cite{W}. Other notable works include Lax-Phillips \cite{LP1}-\cite{LP3}, Babich-Grigoreva \cite{BG}, Filippov-Zayaev  \cite{FZ}, Morawetz-Ralston-Strauss \cite{MRS}, Melrose \cite{Me1}, Lebeau \cite{Le}, Bardos-Lebeau-Rauch \cite{BLR}, Popov \cite{Po}, Harg\'{e}-Lebeau \cite{HL}, Sj\"{o}strand \cite{S2}, Sj\"{o}strand-Zworski \cite{SZ1}-\cite{SZ6} and Stefanov \cite{St}.
See \cite{Me2}, \cite{Z1} for surveys on this topic and other related settings.

\subsection*{Outline of the proof}
Our strategy is based on a modification of the approach in \cite{SZ6} where the phenomenon that resonances appear in bands was first proved. Our paper is organized as follows.

In Section \ref{sec:prelim}, we reduce the problem to the study of an operator constructed from combining a semiclassical differential operator $P-z$ with a boundary operator $\gamma$. The reason we introduce this combined operator is to avoid the domain issues for different Neumann/Robin boundary conditions and treat them in the same setting. Moreover, in the semiclassical setting, the Robin boundary operator is a perturbation of the Neumann boundary operator. We also follow the long tradition of the complex scaling method in mathematical physics, first introduced in \cite{AC}, \cite{BC}, to deform the self-adjoint operator with continuous spectrum to a non-self-adjoint operator whose discrete spectrum near the real axis coincides with the resonances. In our setting, the complex scaling method has been introduced in \cite{SZ1}, and then in \cite{HL}, \cite{SZ4} and \cite{SZ5}.

In Section \ref{sec:model}-\ref{sec:global}, we set up the Grushin problem for the combined operator and therefore identify the resonances with poles of a meromorphic family of operators on the boundary. The survey \cite{SZ7} gives a good reference for the application of Grushin problem in the study of spectral theory; also see the appendix of \cite{HS}.

In Section \ref{sec:model}, we study the model case near the boundary in which case we have an ordinary differential operator with a Neumann boundary operator in the normal direction. This part is the main novelty of this paper. The complication is due to the presence of the boundary operator which makes the total operator not normal. To deal with this, we need a more careful study of the asymptotics of Airy functions in different directions in the complex plane.

In Section \ref{sec:second}, we continue working near the boundary and study the microlocal structure of the Grushin problem. As in \cite{SZ6}, the suitable symbol class for the operators is given by a second microlocalization with respect to the glancing hypersurface. We shall first review the results in \cite[Section 4]{SZ6} for such symbol classes, then see how the operators we construct fit into these classes.

In Section \ref{sec:global}, we combine the work in Section \ref{sec:model} and Section \ref{sec:second} with the results in \cite[Section 7]{SZ6} for the study of the Laplacian operator away from the boundary to set up the global Grushin problem. The construction of the inverse for this Grushin problem is essentially the same as \cite[Section 8]{SZ6} with modification needed for our operator. This produces an effective Hamiltonian $E_{-+}$ which is a matrix-valued operator on the boundary.

Finally in \ref{sec:resfree} we prove the main theorems using the properties of the operator $E_{-+}$.

\subsection*{Acknowledgement}
I would like to thank Maciej Zworski for the encouragement and advice during the preparation of this paper.
Partial support by the National Science Foundation grant DMS-1201417 is also gratefully acknowledged.

\section{Preliminaries and reduction of the problem}
\label{sec:prelim}
We begin by reviewing the definition of the resonances and its multiplicities. Next we apply the standard complex scaling method to identify the resonances with eigenvalues of a non-self-adjoint operator. Then we further reduce the problem to the study of an operator combining this operator with the corresponding boundary operator.

\subsection{Resonances and their multiplicities}
Let us consider different boundary conditions for the Laplacian operator $-\Delta_{\mathbb{R}^n\setminus\mathcal{O}}$ in the exterior of an obstacle $\mathcal{O}$ in $\mathbb{R}^n$:
\begin{equation*}
u|_{\partial\mathcal{O}}=0 \;\;\; \text{ (Dirichlet) }
\end{equation*}
or
\begin{equation}
\label{boundary:NR}
\partial_{\nu}u+\eta u|_{\partial\mathcal{O}}=0 \;\;\;( \text{Neumann when } \eta=0 \text{ or Robin})
\end{equation}
where $\eta\in C^\infty(\partial\mathcal{O};\mathbb{R})$. For the Dirichlet problem, $-\Delta_{\mathbb{R}^n\setminus\mathcal{O}}$ has the natural domain $H_0^1(\mathbb{R}^n\setminus\mathcal{O})\cap H^2(\mathbb{R}^n\setminus\mathcal{O})$. For the Neumann or Robin problem, $-\Delta_{\mathbb{R}^n\setminus\mathcal{O}}$ has the following domain
\begin{equation}
\label{domain:NR}
\mathcal{D}_\eta(\mathbb{R}^n\setminus\mathcal{O}):=
\{u\in H^2(\mathbb{R}^n\setminus\mathcal{O}):
\partial_{\nu}u+\eta u=0\}.
\end{equation}
In either case, the resonance are defined as the poles of the meromorphic extension of the resolvent
\begin{equation*}
R(\zeta)=(-\Delta_{\mathbb{R}^n\setminus\mathcal{O}}-\zeta^2)^{-1}: L^2_{\comp}(\mathbb{R}^n\setminus\mathcal{O})\to L^2_{\loc}(\mathbb{R}^n\setminus\mathcal{O})
\end{equation*}
from the upper half plane $\Im\zeta>0$ to the whole complex plane if $n$ is odd, the logarithmic covering of $\mathbb{C}\setminus\{0\}$ if $n$ is even. The multiplicity of a resonance $\zeta$ is given by
\begin{equation*}
m_{\mathcal{O}}(\zeta)=\rank\oint_{|z-\zeta|=\epsilon}R(z)2zdz
=\tr\frac{1}{2\pi i}\oint_{|z-\zeta|=\epsilon}R(z)2zdz,
\end{equation*}
where $0<\epsilon\ll1$ so that there are no other resonances on the disk $|z-\zeta|\leqslant\epsilon$.

\subsection{Complex Scaling}
The complex scaling method has a long tradition in mathematical physics. It was first introduced by Aguilar-Combes \cite{AC} and Balslev-Combes \cite{BC} in studying the continuous spectrum of Schr\"{o}dinger operators and later proved to be a strong tool in the study of resonances. Sj\"{o}strand and Zworski build up the theory for the case of scattering by a convex obstacle in a series paper \cite{SZ1}, \cite{SZ4} and \cite{SZ5}. We shall adopt the same approach and notations as in \cite{SZ6} and our previous paper \cite{J}.

Let $\mathcal{O}$ be a convex obstacle in $\mathbb{R}^n$ with smooth boundary. We introduce the following normal geodesic coordinates on the exterior domain $\mathbb{R}^n\setminus\mathcal{O}$:
$$x=(x',x_n)\mapsto x'+x_n\vec{n}(x'),\;\;\; x'\in\partial\mathcal{O},\;\;\; x_n=d(x,\partial\mathcal{O}),$$
where $\nu(x')$ is the exterior unit normal vector to $\mathcal{O}$ at $x'$:
$$\nu(x')\in N_{x'}\partial\mathcal{O},\;\;\; \|\nu(x')\|=1.$$
Then
\begin{equation*}
-\Delta_{\mathbb{R}^n\setminus\mathcal{O}}
=D_{x_n}^2+R(x',D_{x'})-2x_nQ(x_n,x',D_{x'})+G(x_n,x')D_{x_n},
\end{equation*}
where $R(x',D_{x'})$, $Q(x_n,x',D_{x'})$ are second order operators on $\partial\mathcal{O}$:
\begin{equation*}
R(x',D_{x'})=-\Delta_{\partial\mathcal{O}}=(\det(g^{ij}))^{1/2}
\sum_{i,j=1}^{n-1}D_{y_i}(\det(g_{ij}))^{1/2}g^{ij}D_{y_j}
\end{equation*}
is the Laplacian with respect to the induced metric $g=(g_{ij})$ on $\partial\mathcal{O}$ and $Q(x',D_{x'})=Q(0,x',D_{x'})$ is of the form
\begin{equation*}
\det(g^{ij})^{1/2}\sum_{i,j=1}^{n-1}D_{y_j'}
(\det(g_{ij}))^{1/2}a_{ij}D_{y_i'}
\end{equation*}
in any local coordinates such that the principal symbol of $Q$ is the second fundamental form of $\partial\mathcal{O}$ lifted by the duality to $T^\ast\partial\mathcal{O}$:
\begin{equation*}
Q(x',\xi')=\sum_{i,j=1}^{n-1}a_{ij}(x')\xi_i\xi_j.
\end{equation*}
Thus the principal curvatures of $\partial\mathcal{O}$ are the eigenvalues of the quadratic form $Q(x',\xi')$ with respect to the quadratic form $R(x',\xi')$.

Now we consider the complex contour given by
\begin{equation*}
\mathbb{R}^n\setminus\mathcal{O}\ni x\mapsto
z=x+i\theta(x)f'(x)\in
\Gamma\subset\mathbb{R}^n\setminus\mathcal{O}+i\mathbb{R}^n,
\end{equation*}
where $f(x)=\frac{1}{2}d(x,\partial\mathcal{O})^2$. When near the boundary, we scale by the angle $\pi/3$ which is first introduced in \cite{HL}:
\begin{equation*}
\frac{1+i\theta(x)}{|1+i\theta(x)|}=e^{i\pi/3},\;\;\; d(x,\partial\mathcal{O})<C^{-1}
\end{equation*}
and then connect to the scaling with a smaller angle $\theta(x)=\theta_0$ near infinity. Whenever there is no confusion, we shall identify $\Gamma$ with $\mathbb{R}^n\setminus\mathcal{O}$ as above and use the normal geodesic coordinates $(x',x_n)$ as coordinates on $\Gamma$. We define $-\Delta_\Gamma$ as the restriction of the holomorphic Laplacian on $\mathbb{C}^n$
\begin{equation*}
-\Delta_z=\sum_{j=1}^nD_{z_j}^2
\end{equation*}
to $\Gamma$. Therefore we have the following expression near the boundary
\begin{equation*}
-\Delta_\Gamma=e^{-2\pi i/3}((D_{x_n})^2+2x_nQ(x_n,x',D_{x'}))+R(x',D_{x'})+F(x_n,x')D_{x_n}.
\end{equation*}
This shows that $\pi/3$ is the correct scaling angle and we get an Airy-type differential operator in the normal direction.

We can also associate the scaled operator with different boundary conditions on $\partial\Gamma=\partial\mathcal{O}$:
\begin{equation*}
u|_{\partial\mathcal{O}}=0 \;\;\; \text{ (Dirichlet) }
\end{equation*}
or
\begin{equation*}
\partial_{\vec{n}}u+e^{\pi i/3}\eta u|_{\partial\mathcal{O}}=0 \;\;\;( \text{Neumann when } \eta=0 \text{ or Robin}).
\end{equation*}
Now for Dirichlet problem, the scaled operator
$-\Delta_\Gamma$ has the natural domain
$H_0^1(\Gamma)\cap H^2(\Gamma)$ and for Neumann or Robin boundary condition $-\Delta_\Gamma$ has the domain
\begin{equation}
\label{scaled:NR}
\mathcal{D}_\eta(\Gamma):=\{u\in H^2(\Gamma):\partial_\nu u+e^{\pi i/3}\eta u|_{\partial\mathcal{O}}=0\}.
\end{equation}

It was shown in \cite{SZ4} that
\begin{prop}
The spectrum of $-\Delta_\Gamma$ is discrete in $-2\theta_0<\arg z<0$ and the resonances of $-\Delta_{\mathbb{R}^n\setminus\mathcal{O}}$ in the sector $-\theta_0<\arg\zeta<0$ are the same as the square root of the eigenvalues of $-\Delta_\Gamma$ with corresponding boundary condition in $-2\theta_0<\arg z<0$. Moreover, they have the same multiplicities:
\begin{equation*}
m_{\mathcal{O}}(\zeta)=m(z):=\tr\frac{1}{2\pi i}\oint_{|\tilde{z}-z|=\epsilon}
(-\Delta_\Gamma-\tilde{z})^{-1}d\tilde{z}
\end{equation*}
where $z=\zeta^2$, $0<\epsilon\ll1$ so that there are no other eigenvalues of $-\Delta_\Gamma$ in $|\tilde{z}-z|\leqslant\epsilon$.
\end{prop}

\subsection{Further reductions}
We work in the semiclassical setting and introduce $P(h):=-h^2\Delta_\Gamma$. Near the boundary, we have the expression
\begin{equation}
P(h)=e^{-2\pi i/3}((hD_{x_n})^2+2x_nQ(x_n,x',hD_{x'};h))
+R(x',hD_{x'};h)+hF(x_n,x')hD_{x_n}.
\end{equation}
Also for $w\in W\Subset(0,\infty)$ and $|\Im z|\leqslant C$, $|\Re z|\ll\delta^{-1}$, we let $P-z=h^{-2/3}(P(h)-w)-z$, so near the boundary,
\begin{equation}
\label{op:scale}
\begin{split}
P-z=&\;e^{-2\pi i/3}(D_t^2+2tQ(h^{2/3}t,x',hD_{x'};h))\\
&+h^{-2/3}(R(x',hD_{x'};h)-w)+F(h^{2/3}t,x')h^{2/3}D_t-z,
\end{split}
\end{equation}
where $t=h^{-2/3}x_n$.

There are certain difficulty in working with Robin boundary conditions with the domain \eqref{domain:NR} or more precisely with the scaled boundary condition \eqref{scaled:NR}. In normal geodesic coordinates introduced above, the domain will change as the function $\eta$ changes and this causes the difficulty in the formulation of the model problem later.

To avoid this issue, notice that in the $t$-coordinates, the condition \eqref{scaled:NR} can be rewritten as
\begin{equation*}
\partial_tu+h^{2/3}ku|_{t=0}=0,
\end{equation*}
where $k=e^{\pi i/3}\eta$. Roughly speaking, the principal term corresponds to the Neumann boundary condition. This motivates us to consider the Robin boundary problem with general $\eta\in C^\infty(\partial\mathcal{O})$ as a perturbation of the Neumann boundary problem. To achieve this, we shall combine our differential operator $P-z$ with the boundary operator and consider
\begin{equation}
\label{op:comb}
\left(
  \begin{array}{c}
    P-z \\
    \gamma \\
  \end{array}
\right): H^2(\mathbb{R}^n\setminus\mathcal{O})\to L^2(\mathbb{R}^n\setminus\mathcal{O})\times H^l(\partial\mathcal{O})
\end{equation}
where for Dirichlet problem, $l=\frac{3}{2}$,
\begin{equation*}
\gamma=\gamma_0:H^2(\mathbb{R}^n\setminus\mathcal{O})\to H^{3/2}(\partial\mathcal{O}),\;\;\; u\mapsto u|_{\partial\mathcal{O}};
\end{equation*}
and for Neumann or Robin problem $(k=e^{\pi i/3}\gamma)$ that we shall focus on, $l=\frac{1}{2}$,
\begin{equation}
\label{op:NR}
\gamma=h^{2/3}(\gamma_1+k\gamma_0):
H^2(\mathbb{R}^n\setminus\mathcal{O})\to H^{1/2}(\partial\mathcal{O}),\;\;\; u\mapsto h^{2/3}(\partial_\nu u+ku)|_{\partial\mathcal{O}}.
\end{equation}
In the coordinates $(t,x')$, we have $\gamma(u)=u(0,\cdot)$ (Dirichlet) or
\begin{equation*}
\gamma(u)=(\partial_tu+h^{2/3}ku)(0,\cdot)\;\; \text{ (Neumann or Robin).}
\end{equation*}
Therefore from now on we shall think of $P-z$ as the first component of the combined operator \eqref{op:comb}, i.e. the differential operator from $H^2(\mathbb{R}^n\setminus\mathcal{O})$ to $L^2(\mathbb{R}^n\setminus\mathcal{O})$ instead of an operator with a smaller domain \eqref{scaled:NR}. Moreover, to avoid confusion, we shall write $R_P(z)$ to be the resolvent of $P$ with domain \eqref{scaled:NR}, or in other words, $R_P(z)$ is a right inverse of $P-z:H^2\to L^2$ satisfying $\gamma R_P(z)=0$. We wish to use our new operator \eqref{op:comb} to give an equivalent description of resonances instead of
\begin{equation}
\label{res:scale}
m(h^{-2}(w+h^{2/3}z))=\tr\frac{1}{2\pi i}\oint_{|\tilde{z}-z|=\epsilon}
R_P(\tilde{z})d\tilde{z},\;\;\; 0<\epsilon\ll1.
\end{equation}

\begin{prop}
The eigenvalues of $P$ are exactly the poles of
\begin{equation}
\label{op:invcomb}
\left(
  \begin{array}{c}
    P-z \\
    \gamma \\
  \end{array}
\right)^{-1}: L^2(\mathbb{R}^n\setminus\mathcal{O})\times H^l(\partial\mathcal{O})\to H^2(\mathbb{R}^n\setminus\mathcal{O})
\end{equation}
as a meromorphic operator-valued function in $z$. Moreover, they have the same multiplicity:
\begin{equation}
\label{res:comb}
m(h^{-2}(w+h^{2/3}z))=\tr-\frac{1}{2\pi i}\oint_{|\tilde{z}-z|=\epsilon}
\left(
  \begin{array}{c}
    P-\tilde{z} \\
    \gamma \\
  \end{array}
\right)^{-1}
\frac{d}{d\tilde{z}}\left(
  \begin{array}{c}
    P-\tilde{z} \\
    \gamma \\
  \end{array}
\right)d\tilde{z},
\end{equation}
where $0<\epsilon\ll1$ is chosen in a way that there are no other poles for the operator \eqref{op:invcomb} in $|\tilde{z}-z|<\epsilon$.
\end{prop}
\begin{proof}
Let $K$ be a right inverse of $\gamma$:
\begin{equation}
\label{op:invtrace}
K:L^2(\partial\mathcal{O})\to H^2(\mathbb{R}^n\setminus\partial\mathcal{O}),\;\;\; \gamma Kg=g,\;\;\; \forall g\in H^l(\partial\mathcal{O}).
\end{equation}
One possible choice is the so-called Poisson operator, but any choice will be good for us. Then we have
\begin{equation}
\label{rel:comb}
\left(
  \begin{array}{c}
    P-z \\
    \gamma \\
  \end{array}
\right)^{-1}=(R_P(z),K-R_P(z)(P-z)K),
\end{equation}
In fact, for any $(v,g)\in L^2(\mathbb{R}^n\setminus\mathcal{O})\times H^l(\partial\mathcal{O})$, let
$$u=R_P(z)v+(K-R_P(z)(P-z)K)g,$$
by the construction of $K$, \eqref{op:invtrace}, and the fact that $\gamma R_P(z)=0$,
\begin{equation*}
(P-z)u=v+(P-z)Kg-(P-z)Kg=v,\;\;\; \gamma u=\gamma Kg=g.
\end{equation*}
Therefore \eqref{rel:comb} gives
\begin{equation*}
\left(
  \begin{array}{c}
    P-z \\
    \gamma \\
  \end{array}
\right)^{-1}
\frac{d}{dz}\left(
  \begin{array}{c}
    P-z \\
    \gamma \\
  \end{array}
\right)=(R_P(z),K-R_P(z)(P-z)K)
\left(
  \begin{array}{c}
    -1 \\
    0 \\
  \end{array}
\right)
=-R_P(z).
\end{equation*}
Now \eqref{res:comb} and the proposition follows directly from \eqref{res:scale}.
\end{proof}

In this paper, we shall work with the Neumann/Robin boundary \eqref{boundary:NR} condition. The techniques here can certainly be applied to Dirichlet boundary condition. However, in the Dirichlet case, since the domain is already simple enough, we do not need this reduction and a direct approach without the boundary operator is given in \cite{SZ6}.

\subsection{A simple model}
We conclude this section by presenting a simple model motivating our approach to boundary value problems using a Grushin reduction for an operator combining a differential operator and a boundary operator.

We consider the differential operator $P=-\frac{d^2}{dx^2}$ with Neumann boundary condition on the interval $[0,\pi]$. The spectrum of the operator is discrete: $\sigma(P)=\{\lambda_k=k^2: k=0,1,2,\ldots\}$
and each eigenspace is one-dimensional:
\begin{equation*}
E_k=\{f\in H^2[0,1]|f'(0)=f'(1)=0, -f''=\lambda_kf\}=\mathbb{C}\cos kx.
\end{equation*}

We set up a Grushin problem to capture the first $m$ eigenvalues using a finite matrix. For simplicity, let us consider the case $m=1$
so that the first eigenvalue is $\lambda_0=0$ with unit eigenvector $e_0=\frac{1}{\pi}$. Put
\begin{equation*}
\left(
  \begin{array}{cc}
    P-z & R_- \\
    R_+ & 0 \\
  \end{array}
\right):\mathcal{D}\times\mathbb{C}\to L^2[0,\pi]\times\mathbb{C},
\end{equation*}
where
\begin{equation*}
\mathcal{D}=\{u\in H^2[0,\pi]:u'(0)=u'(\pi)=0\}
\end{equation*}
and
\begin{equation*}
Pu=-u'',\;\;\;R_+u=\langle u,e_0\rangle=\frac{1}{\pi}\int_0^\pi udx,
\;\;\;R_-u_-=u_-e_0=\frac{u_-}{\pi}.
\end{equation*}
Then
\begin{equation*}
\left(
  \begin{array}{cc}
    P-z & R_- \\
    R_+ & 0 \\
  \end{array}
\right)
\left(
  \begin{array}{c}
    u \\
    u_- \\
  \end{array}
\right)=
\left(
  \begin{array}{c}
    v \\
    v_+ \\
  \end{array}
\right)
\end{equation*}
is equivalent to
\begin{equation*}
-u''-zu+\frac{u_-}{\pi}=v,\;\;\;\frac{1}{\pi}\int_0^\pi udx=v_+.
\end{equation*}
We can integrate the first equation on $[0,\pi]$ to get
\begin{equation*}
-(u'(\pi)-u'(0))-z\int_0^\pi udx+u_-=\int_0^\pi vdx
\end{equation*}
and thus
\begin{equation*}
u_-=(u'(\pi)-u'(0))+z\int_0^\pi udx+\int_0^\pi vdx=\pi zv_++\int_0^\pi vdx.
\end{equation*}
It is then not difficult to see that for $z<1$, we can use this $u_-$ to solve $u$ uniquely. Therefore the Grushin problem is well-posed with inverse
\begin{equation*}
\left(
  \begin{array}{cc}
    E & E_+ \\
    E_- & E_{-+} \\
  \end{array}
\right):L^2[0,\pi]\times\mathbb{C}\to\mathcal{D}\times\mathbb{C},
\end{equation*}
which has an explicit expression and we have seen that $E_{-+}=\pi z$ which is invertible if and only if $z\neq\lambda_0=0$.

The situation is somewhat similar to our case of obstacle scattering if we regard the left end point $x=0$ as the boundary, and the right end point $x=\pi$ as infinity. Recall that in the case of obstacle scattering, since the outgoing condition becomes $L^2$-condition after complex scaling, we get a ``boundary condition'' at infinity. Now, we consider another Grushin problem for $-\frac{d^2}{dx^2}$, or rather the following operator
\begin{equation*}
\left(
  \begin{array}{c}
    -\frac{d^2}{dx^2}-z \\
    \gamma_1 \\
  \end{array}
\right):\mathcal{D}'=\{u\in H^2[0,1]|u'(\pi)=0\}\to L^2[0,1]\times\mathbb{C}
\end{equation*}
where $\gamma_1u=u'(0)$. We use the same $R_+$ and $R_-$ as above to construct the Grushin problem
\begin{equation*}
\left(
  \begin{array}{cc}
    -\frac{d^2}{dx^2}-z & R_- \\
    \gamma_1 & 0\\
    R_+ & 0 \\
  \end{array}
\right):\mathcal{D}'\times\mathbb{C}\to L^2[0,\pi]\times\mathbb{C}\times\mathbb{C}.
\end{equation*}
Now
\begin{equation*}
\left(
  \begin{array}{cc}
    -\frac{d^2}{dx^2}-z & R_- \\
    \gamma_1 & 0\\
    R_+ & 0 \\
  \end{array}
\right)
\left(
  \begin{array}{c}
    u \\
    u_- \\
  \end{array}
\right)=
\left(
  \begin{array}{c}
    v \\
    v_0 \\
    v_+ \\
  \end{array}
\right)
\end{equation*}
is equivalent to
\begin{equation*}
-u''-zu+\frac{u_-}{\pi}=v,\;\;\;u'(0)=v_0,\;\;\;\frac{1}{\pi}\int_0^\pi udx=v_+.
\end{equation*}
Again, integrating the first equation gives
\begin{equation*}
-(u'(\pi)-u'(0))-z\int_0^\pi udx+u_-=\int_0^\pi vdx
\end{equation*}
and thus
\begin{equation*}
u_-=(u'(\pi)-u'(0))+z\int_0^\pi udx+\int_0^\pi vdx=-v_0+\pi zv_++\int_0^\pi vdx.
\end{equation*}
Again, using this $u_-$, it is not difficult to solve $u$ uniquely for $z<1$. Hence this Grushin problem is also well-posed with inverse
\begin{equation*}
\left(
  \begin{array}{ccc}
    E & K & E_+ \\
    E_- & K_- & E_{-+} \\
  \end{array}
\right):L^2[0,\pi]\times\mathbb{C}\times\mathbb{C}
\to\mathcal{D}'\times\mathbb{C},
\end{equation*}
which again has an explicit expression. We find that $E_{-+}=\pi z$ coincides with $E_{-+}$ we found in the previous Grushin problem.

Of course in this trivial example we can compute everything explicitly without Grushin reduction. The importance of the Grushin problem is that we can perturb the operator and the invertibility of the perturbed operator is captured by the finite matrix $E_{-+}$ (in our case it is a $1\times1$ matrix, i.e. a scalar.) This reduces the infinite-dimensional problem to a finite-dimensional one. The second Grushin problem also allows us to perturb the boundary condition at $0$ which turns out to be crucial in our setting.

\section{Model Grushin problems}
\label{sec:model}

In this section, we shall study the model problem for ordinary differential operators by setting up a suitable Grushin problem.
Recall that we have the combined operator \eqref{op:comb}
\begin{equation*}
\left(
  \begin{array}{c}
    P-z \\
    \gamma \\
  \end{array}
\right): H^2(\mathbb{R}^n\setminus\mathcal{O})\to L^2(\mathbb{R}^n\setminus\mathcal{O})\times H^{1/2}(\partial\mathcal{O}),
\end{equation*}
where $P-z$ is given by
\begin{equation*}
P-z=h^{-2/3}(-h^2\Delta_\Gamma-w)-z:
H^2(\mathbb{R}^n\setminus\mathcal{O})\to L^2(\mathbb{R}^n\setminus\mathcal{O})
\end{equation*}
and $\gamma$ is given by
\begin{equation*}
\gamma=h^{2/3}(\gamma_1+k\gamma_0):
H^2(\mathbb{R}^n\setminus\mathcal{O})\to L^2(\partial\mathcal{O}),\;\;\; u\mapsto h^{2/3}(\partial_\nu u+ku)|_{\partial\mathcal{O}}.
\end{equation*}
In local coordinates $(t=h^{-2/3}x_n,x')$ near the boundary introduced in Section \ref{sec:prelim}, we have
\begin{equation*}
\begin{split}
P-z=&\;e^{-2\pi i/3}(D_t^2+2tQ(h^{2/3}t,x',hD_{x'};h))\\
&+h^{-2/3}(R(x',hD_{x'};h)-w)+F(h^{2/3}t,x')h^{2/3}D_t-z,
\end{split}
\end{equation*}
and
\begin{equation*}
\gamma(u)=\gamma_1(u)+h^{2/3}k\gamma_0(u)
=(\partial_tu+h^{2/3}ku)(0,\cdot).
\end{equation*}

Therefore we start by ignoring the lower order terms and considering a model operator
\begin{equation}
\label{model:airy}
P_\lambda-z=e^{-2\pi i/3}(D_t^2+\mu t)+\lambda-z
\end{equation}
with $\gamma_1:u\mapsto u'(0)$, where $\lambda\in\mathbb{R}$, $C^{-1}\leqslant\mu\leqslant C$ and $|\Im z|<C_1$ with
$C_1$ large but fixed. Here we regard $\lambda$ as $h^{-2/3}(R(x',hD_{x'})-w)$, and $\mu$ as $Q(0,x',hD_{x'})$. Other terms will be small perturbation.

The model above is only necessary for handling the region near the glancing hypersurface $\Sigma_w=\{R(x',\xi')=w\}$. In the situation that $|\lambda|\gg1+|\Re z|$, i.e. away from the glancing region, since $Q$ is bounded by $R$, we can also treat the term $e^{-2\pi i/3}\mu t$ as a perturbation and instead consider the model operator
\begin{equation}
\label{model:easy}
P_\lambda^\#-z=e^{-2\pi i/3}D_t^2+\lambda-z
\end{equation}
with the same $\gamma_1$ and $\lambda\in\mathbb{R}, |\Im z|<C_1$. Here we note that \eqref{model:easy} is elliptic as $|\lambda-\Re z|\gg1$ and thus this model is easier to work with.

In this section, we shall first review some properties of Airy function and estimates of Airy operators and boundary operators. Next we solve the Grushin problem for the model Airy operators in the case $\mu=1$. Then we treat the easier model operator \eqref{model:easy} in the same way. Finally we shall show how the additional parameter $\mu$ affects our construction and that all the estimates are uniform for $\mu$ in a compact subset of $(0,\infty)$.

\subsection{Asymptotics and zeroes of Airy functions}
Recall that the Airy function $\Ai$ can be defined by the formula
\begin{equation}
\label{fn:Airy}
\Ai(t)=\frac{1}{2\pi}\int_{\Im\sigma=\delta>0}e^{i(\sigma^3/3)+i\sigma t}d\sigma
\end{equation}
in the real domain and it is in fact an entire function for $t\in\mathbb{C}$ with different asymptotic behaviors in different directions. For example, in the positive real direction,
\begin{equation}
\label{asy:Airy-pos}
\begin{split}
\Ai(t)=&\;(2\sqrt{\pi})^{-1}t^{-1/4}
e^{-\frac{2}{3}t^{3/2}}(1+O(t^{-3/2})),\\
\Ai'(t)=&\;-(2\sqrt{\pi})^{-1}t^{1/4}
e^{-\frac{2}{3}t^{3/2}}(1+O(t^{-3/2})),\\
\end{split}
\end{equation}
as $t\to\infty$; while in the negative real direction,
\begin{equation}
\label{asy:Airy-neg}
\begin{split}
\Ai(-t)=&\;\pi^{-1/2}t^{-1/4}
\left(\sin(\frac{2}{3}t^{3/2}+\frac{\pi}{4})+O(t^{-3/2})\right),\\
\Ai'(-t)=&\;-\pi^{-1/2}t^{1/4}
\left(\cos(\frac{2}{3}t^{3/2}+\frac{\pi}{4})+O(t^{-3/2})\right),\\
\end{split}
\end{equation}
as $t\to\infty$. Moreover, \eqref{asy:Airy-pos} holds away from the negative real axis:
\begin{equation}
\label{asy:Airy-com}
\begin{split}
\Ai(z)=&\;(2\sqrt{\pi})^{-1}
e^{-\zeta}z^{-1/4}(1+O(|\zeta|^{-1})),\\
\Ai'(z)=&\;-(2\sqrt{\pi})^{-1}
e^{-\zeta}z^{1/4}(1+O(|\zeta|^{-1})),\\
\end{split}
\end{equation}
uniformly for $0\leqslant|\arg z|\leqslant\pi-\delta$, where $\delta>0$ is fixed. Here $\zeta=\frac{2}{3}z^{3/2}$ and we choose the branch such that if $z$ is real and positive, then so is $\zeta$.

Let $0<\zeta_1<\zeta_2<\cdots$ and $0<\zeta_1'<\zeta_2'<\cdots$ be the negatives of the zeroes of $\Ai$ and $\Ai'$, respectively. All of these zeroes are simple and we have $\zeta_j'<\zeta_j<\zeta_{j+1}'$. The distances between the zeroes get closer: $\zeta_{j+1}-\zeta_j\searrow0$ and $\zeta_{j+1}'-\zeta_j'\searrow0$ as $j\to\infty$. This can be proved by Sturm's comparison theorem.

The Airy function $\Ai$ solves the simple differential equation
$(D_t^2+t)\Ai(t)=0, t\in\mathbb{R}$.
Therefore all the eigenfunctions and eigenvalues for the Dirichlet and Neumann realization of the Airy operator $D_t^2+t$ on $[0,\infty)$ are given by translations of the Airy function:
\begin{equation*}
(D_t^2+t)\Ai(t-\zeta_j)=\zeta_j\Ai(t-\zeta_j),\;\;\;
(D_t^2+t)\Ai(t-\zeta_j')=\zeta_j'\Ai(t-\zeta_j').
\end{equation*}
Since we are only working with Neumann boundary condition, let us write $e_j(t)=c_j\Ai(t-\zeta_j')$ to be the normalized eigenfunctions of the Neumann realization of $D_t^2+t$ on $(0,\infty)$. Then $\{e_j\}_{j=1}^\infty$ forms an orthonormal basis for $L^2(0,\infty)$.

\subsection{Some basic estimates}
In this part, we give some elementary estimates on Airy operators and the boundary operators, some of these estimates can be found in \cite{J}.

Consider the Airy operator
$D_t^2+t:B\subset L^2\to L^2$
and the boundary operators
\begin{equation*}
\gamma_0:B\to\mathbb{C},\;\; u\mapsto u(0),\;\;\;
\gamma_1:B\to\mathbb{C},\;\; u\mapsto u'(0).
\end{equation*}
Here $L^2=L^2(0,\infty)$ and $B=\{u\in L^2:D_t^2u, tu\in L^2\}$ is a Banach space equipped with the norm
\begin{equation}
\label{norm:b}
\|u\|_B=\|D_t^2u\|+\|tu\|+\|u\|,
\end{equation}
where we use $\|\cdot\|$ to represent the standard $L^2$-norm on $(0,\infty)$.

It is clear that $\|(D_t^2+t)u\|\leqslant C\|u\|_B$. More precisely, we have the following identity,
\begin{equation}
\label{id:airy}
\|(D_t^2+t)u\|^2=\|D_t^2u\|^2+\|tu\|^2
+2\|\sqrt{t}D_tu\|^2-|\gamma_0u|^2,
\end{equation}
for any $u\in C_0^\infty([0,\infty))$. The proof is bases on a simple integration by parts. To see this, let $\langle,\rangle$ be the standard $L^2$ inner product on $(0,\infty)$. Then
\begin{equation*}
\begin{split}
\|(D_t^2+t)u\|^2=&\;\|D_t^2u\|^2+\|tu\|^2+2\Re\langle D_t^2u, tu\rangle\\
=&\;\|D_t^2u\|^2+\|tu\|^2+2\Re\langle D_tu,D_t(tu)\rangle\\
=&\;\|D_t^2u\|^2+\|tu\|^2+2\Re\langle D_tu,tD_tu\rangle
+2\Re\frac{1}{i}\langle D_tu,u\rangle\\
=&\;\|D_t^2u\|^2+\|tu\|^2
+2\|\sqrt{t}D_tu\|^2-|\gamma_0u|^2.
\end{split}
\end{equation*}
Here in the last step, we use again the integration by parts
\begin{equation}
\label{id:gamma0}
\langle D_tu,u\rangle=\langle u,D_tu\rangle-i|u(0)|^2
\end{equation}
to get
\begin{equation*}
\Re\frac{1}{i}\langle D_tu,u\rangle=\Im\langle D_tu,u\rangle=-\frac{i}{2}|u(0)|^2.
\end{equation*}

Next we give some estimates of $\gamma_0$ and $\gamma_1$. For any $u\in C^\infty_0([0,\infty))$, by the Cauchy-Schwartz inequality and \eqref{id:gamma0}, we get
\begin{equation*}
|\gamma_0u|^2\leqslant2\|D_tu\|\|u\|,
\end{equation*}
and similarly
\begin{equation*}
|\gamma_1u|^2\leqslant2\|D_t^2u\|\|D_tu\|.
\end{equation*}
Another application of integration by parts and the Cauchy-Schwartz inequality also gives
\begin{equation*}
\begin{split}
\|D_tu\|^2=&\;\langle D_t^2u,u\rangle-u(0)u'(0)\\
\leqslant&\;|\gamma_1u||\gamma_0u|+\|D_t^2u\|\|u\|\\
\leqslant&\;2\|D_t^2u\|^{1/2}\|u\|^{1/2}\|D_tu\|+\|D_t^2u\|\|u\|
\end{split}
\end{equation*}
which leads to the standard interpolation estimates
\begin{equation}
\label{es:interpolation}
\|D_tu\|\leqslant(\sqrt{2}+1)\|D_t^2u\|^{1/2}\|u\|^{1/2}.
\end{equation}
As a consequence, for any $\epsilon>0$,
\begin{equation}
\label{es:boundary}
\begin{split}
|\gamma_0u|\leqslant&\;C\|D_t^2u\|^{1/4}\|u\|^{3/4}\leqslant \epsilon\|D_t^2u\|+C_\epsilon\|u\|\\
|\gamma_1u|\leqslant&\;C\|D_t^2u\|^{3/4}\|u\|^{1/4}\leqslant \epsilon\|D_t^2u\|+C_\epsilon\|u\|.
\end{split}
\end{equation}
Now from \eqref{norm:b} and \eqref{id:airy} we get
\begin{equation}
\label{eq:normairy}
\|u\|_B\leqslant C(\|u\|_{L^2}+\|(D_t^2+t)u\|_{L^2})
\end{equation}
and
\begin{equation}
\label{esB:boundary}
|\gamma_0u|\leqslant C\|u\|_B,\;\;\; |\gamma_1u|\leqslant C\|u\|_B.
\end{equation}

We finish this part by using these two estimates to show that elements in $B$ can be written in a unique way as a linear combination of the Neumann Airy eigenfunctions $(e_j)_{j=1}^\infty$ introduced in the previous section and one other element $f\in B$ with $\gamma_1f\neq0$. We remark that $(e_j)$ is not an orthonormal basis in $B$, so this expression might be different from the orthogonal expansion in $L^2$.

On one hand, if the sum $\sum_ju_je_j$ converges in $B$ to some $u$, then by \eqref{esB:boundary} we have $\gamma_1u=\sum_ju_j\gamma_1e_j=0$. On the other hand, if $u\in B$ satisfies $\gamma_1u=u'(0)=0$, then we can consider the $L^2$-orthogonal expansion
\begin{equation}
\label{ex:ortho}
u=\sum_j\langle u,e_j\rangle e_j.
\end{equation}
By \eqref{eq:normairy}, we have for any finite subset $J$ of $\mathbb{Z}_+$,
\begin{equation*}
\begin{split}
\|\sum_{j\in J}\langle u,e_j\rangle e_j\|_B\leqslant&\; C(\|\sum_{j\in J}\langle u, e_j\rangle e_j\|+\|(D_t^2+t)\sum_{j\in J}
\langle u, e_j\rangle e_j\|)\\
\leqslant&\; C(\|\sum_{j\in J}\langle u, e_j\rangle e_j\|+\|\sum_{j\in J}\zeta_j'\langle u, e_j\rangle e_j\|)\\
\leqslant&\; C(\|\sum_{j\in J}\langle u, e_j\rangle e_j\|+\|\sum_{j\in J}\langle u, (D_t^2+t)e_j\rangle e_j\|)\\
\leqslant&\; C(\|\sum_{j\in J}\langle u, e_j\rangle e_j\|+\|\sum_{j\in J}\langle (D_t^2+t)u, e_j\rangle e_j\|).\\
\end{split}
\end{equation*}
which shows that the sum \eqref{ex:ortho} converges to $u$ in $B$ since $(D_t^2+t)u\in L^2$.

Therefore if we fix some $f\in B$ such that $\gamma_1f=f'(0)\neq0$, then every $u\in B$ can be uniquely expressed in the form
\begin{equation}
\label{ex:B}
u=u_0f+\sum_{j=1}^\infty u_je_j
\end{equation}
where the sum converges in $B$. We simply choose $u_0$ first such that $\gamma_1(u-u_0f)=0$, then write the orthogonal expansion of $u-u_0f$ by $(e_j)$ in $L^2$, i.e. $u_j=\langle u-u_0f,e_j\rangle$.

\subsection{Model Airy problem}
The operator in \eqref{model:airy} (taking $\mu=1$) combining with the Neumann boundary operator
\begin{equation}
\label{model:airycomb}
\left(
  \begin{array}{c}
    P_\lambda-z \\
    \gamma_1 \\
  \end{array}
\right):B\to L^2\times\mathbb{C}
\end{equation}
may not be invertible for all $z$ with $|\Im z|<C_1$. In fact,
let us take $N=N(C_1)$ as the largest number such that
\begin{equation*}
|\Im e^{-2\pi i/3}\zeta'_N|\leqslant C_1,
\end{equation*}
so that $e^{-2\pi i/3}\zeta'_j+\lambda-z\neq0$ for $j\geqslant N+1$.
Then \eqref{model:airycomb} is not invertible precisely when $e^{-2\pi i/3}\zeta'_j+\lambda-z=0$ for some $j=1,\ldots,N$ since $e_j$ is in its kernel. Therefore we need to correct this operator in a suitable way to make it invertible. We shall also modify our spaces by putting an exponential weight. Moreover, we also need to add correct powers of $\langle\lambda-\Re z\rangle$ in the norm.

More precisely, let us consider the following Grushin problem for \eqref{model:airycomb}:
\begin{equation}
\label{model:airygrushin}
\mathcal{P}_\lambda(z)=\left(
                         \begin{array}{cc}
                           P_\lambda-z & R_-^0 \\
                           \gamma_1 & r_- \\
                           R_+^0 & 0 \\
                         \end{array}
                       \right):
\mathcal{B}_{z,\lambda,r}\to\mathcal{H}_{z,\lambda,r}
\end{equation}
(Later on we shall always choose $r_-=0$.) Here the spaces and the norms on the spaces are given by
\begin{equation}
\label{model:airyspace}
\begin{split}
\mathcal{B}_{z,\lambda,r}&\;=B_{z,\lambda,r}\times\mathbb{C}^N,\\
\left\|\left(
         \begin{array}{c}
           u \\
           u_- \\
         \end{array}
       \right)
\right\|_{\mathcal{B}_{z,\lambda,r}}&\;
=\|u\|_{B_{z,\lambda,r}}+|u_-|,\\
\mathcal{H}_{z,\lambda,r}&\;=L^2_r\times
\mathbb{C}_{\langle\lambda-\Re z\rangle^{1/4}}\times
\mathbb{C}^N_{\langle\lambda-\Re z\rangle},\\
\left\|\left(
         \begin{array}{c}
           v \\
           v_0 \\
           v_+
         \end{array}
       \right)
\right\|_{\mathcal{H}_{z,\lambda,r}}&\;=\|v\|_{L^2_r}
+\langle\lambda-\Re z\rangle^{1/4}|v_0|
+\langle\lambda-\Re z\rangle|v_+|.
\end{split}
\end{equation}
with $|\cdot|$ fixed norms on $\mathbb{C}$ or $\mathbb{C}^N$ and $L^2_r=L^2([0,\infty),e^{rt}dt),
B_{z,\lambda,r}=\{u\in L^2_r; D_t^2u, tu\in L^2_r\}$. The norms are given by the standard weighted $L^2$-norm $\|\cdot\|_{L^2_r}$ and
\begin{equation}
\label{norm:weighted}
\|u\|_{B_{z,\lambda,r}}=\langle\lambda-\Re z\rangle\|u\|_{L^2_r}
+\|D_t^2u\|_{L^2_r}+\|tu\|_{L^2_r},
\end{equation}
respectively. Moreover, the operators are given by
\begin{equation*}
\begin{split}
P_\lambda-z&: B_r\to L^2_r,\;\;\; u\mapsto(e^{-2\pi i/3}(D_t^2+t)+\lambda-z)u; \\
\gamma_1&: B_r\to\mathbb{C},\;\;\; u\mapsto u'(0);\\
R_+^0&: B_r\to\mathbb{C}^N,\;\;\; u\mapsto (\langle u,e_j\rangle)_{1\leqslant j\leqslant N};\\
R_-^0&: \mathbb{C}^N\to L^2_r,\;\;\; u_-\mapsto\sum_{j=1}^Nu_-(j)e_j;\\
r_-&: \mathbb{C}^N\to\mathbb{C},\;\;\; u_-\mapsto\sum_{j=1}^Nr_ju_-(j).
\end{split}
\end{equation*}

We remark that the heuristic reason for the weight $\langle\lambda-\Re z\rangle^{1/4}$ in the second component $\mathbb{C}$ on $\mathcal{H}_{z,\lambda,r}$ is that $\langle\lambda-\Re z\rangle$ roughly represents the Laplacian on the boundary $\langle\Delta_{\partial\mathcal{O}}\rangle$ (up to some parameters). Therefore if $u\in H^2(\mathbb{R}^n\setminus\mathcal{O})$, then by the well-known property of boundary operators $\partial_\nu u|_{\partial\mathcal{O}}\in H^{1/2}(\partial\mathcal{O})$ the norm of which corresponds to $\langle\lambda-\Re z\rangle^{1/4}$. We can also see that this is the correct weight by rescaling the estimate \eqref{es:boundary}. For the same reason, if we wish to work with Dirichlet boundary operator, then we need to replace this weight $\langle\lambda-\Re z\rangle^{1/4}$ by $\langle\lambda-\Re z\rangle^{3/4}$.

Moreover, to handle powers of $t$ which will appear in lower order terms, it is necessary to introduce the exponential weight $e^{rt}, r>0$ in the definition of spaces $\mathcal{B}_{z,\lambda,r}$ and $\mathcal{H}_{z,\lambda,r}$. This will be explained in full details in the next section.

For $r=0$, it is clear that the space $B_{z,\lambda,0}$ is just $B$ in the previous section with an equivalent norm (of course not uniformly in $z,\lambda$) and $\mathcal{P}_\lambda(z):
\mathcal{B}_{z,\lambda,0}\to\mathcal{H}_{z,\lambda,0}$ is a uniformly bounded operator. Now we look for the inverse of $\mathcal{P}_\lambda(z)$. Let
\begin{equation}
\label{eq:grushin}
\mathcal{P}_\lambda(z)\left(
                        \begin{array}{c}
                          u \\
                          u_- \\
                        \end{array}
                      \right)
                      =\left(
                        \begin{array}{c}
                          v \\
                          v_0 \\
                          v_+\\
                        \end{array}
                      \right).
\end{equation}
Then explicitly we have
\begin{equation*}
\begin{split}
(P_\lambda-z)u+R_-^0u_-=&\;v\\
u'(0)+r_-u_-=&\;v_0\\
R_+^0u=&\;v_+.
\end{split}
\end{equation*}
We express $v$ in terms of the orthonormal basis $(e_j)_{j=1}^\infty$ in $L^2$:
\begin{equation*}
v=\sum_{j=1}^\infty v_je_j,
\end{equation*}
and we write $v_+=(v_+(j))_{1\leqslant j\leqslant N}$. Then we look for solutions with $u\in B$ as in \eqref{ex:B}
\begin{equation*}
u=u_0f+\sum_{j=1}^\infty u_je_j
\end{equation*}
and
\begin{equation*}
u_-=(u_-(j))_{1\leqslant j\leqslant N}.
\end{equation*}
Let us write
\begin{equation*}
f_0:=f'(0),\;\;\;f_j:=\langle f,e_j\rangle, \;\;\; \eta_j:=e^{-2\pi i/3}\zeta_j'+\lambda-z.
\end{equation*}
then we have
\begin{equation*}
(P_\lambda-z)e_j=\eta_je_j,\;\;\; (P_\lambda-z)^\ast e_j=\bar{\eta}_je_j.
\end{equation*}
where $(P_\lambda-z)^\ast=e^{2\pi i/3}(D_t^2+t)+\lambda-\bar{z}$ is the formal adjoint of $P_\lambda-z$. Moreover,
\begin{equation*}
\langle(P_\lambda-z)f,e_j\rangle=e^{-2\pi i/3}e_j(0)f_0+\langle f,(P_\lambda-z)^\ast e_j\rangle=e^{-2\pi i/3}e_j(0)f_0+\eta_jf_j.
\end{equation*}
Then we can rewrite the system \eqref{eq:grushin} as an infinite system of linear equations:
\begin{equation}
\label{eq:grushin2}
\begin{split}
[e^{-2\pi i/3}e_j(0)f_0+\eta_jf_j]u_0+\eta_ju_j+u_-(j)=&\;v_j,\;\; (1\leqslant j\leqslant N)\\
[e^{-2\pi i/3}e_j(0)f_0+\eta_jf_j]u_0+\eta_ju_j=&\;v_j,\;\; (j\geqslant N+1)\\
f_0u_0+\sum_{j=1}^Nr_ju_-(j)=&\;v_0\\
f_ju_0+u_j=&\;v_+(j),\;\; (1\leqslant j\leqslant N).\\
\end{split}
\end{equation}
It is not difficult to see that as long as
\begin{equation*}
1-e^{-2\pi i/3}\sum_{j=1}^Nr_je_j(0)\neq0,
\end{equation*}
we have a unique solution for \eqref{eq:grushin2},
\begin{equation*}
\begin{split}
u_0=&\;\left[1-e^{-2\pi i/3}\sum_{j=1}^Nr_je_j(0)\right]^{-1}
f_0^{-1}\left[v_0+\sum_{j=1}^Nr_j(\eta_jv_+(j)-v_j)\right]\\
u_j=&\;v_+(j)-f_ju_0, \;\;\;(1\leqslant j\leqslant N)\\
u_j=&\;\eta_j^{-1}(v_j-(e^{-2\pi i/3}e_j(0)f_0+\eta_jf_j)u_0), \;\;\;(j\geqslant N+1)\\
u_-(j)=&\;v_j-\eta_jv_+(j)-e^{-2\pi i/3}e_j(0)f_0u_0, \;\;\;(1\leqslant j\leqslant N).
\end{split}
\end{equation*}
For simplicity, henceforth we shall choose $f_0=1,r_-=0$ (though other choices are also possible). Then the solution becomes
\begin{equation}\label{eq:solution}
\begin{split}
u_0=&\;v_0\\
u_j=&\;v_+(j)-f_jv_0, \;\;\;(1\leqslant j\leqslant N)\\
u_j=&\;\eta_j^{-1}(v_j-e^{-2\pi i/3}e_j(0)v_0)-f_jv_0, \;\;\;(j\geqslant N+1)\\
u_-(j)=&\;v_j-e^{-2\pi i/3}e_j(0)v_0-\eta_jv_+(j),
\;\;\;(1\leqslant j\leqslant N).
\end{split}
\end{equation}

Now we need to estimate the norm.
\begin{lem}
\label{lem:airy1}
The Grushin problem \eqref{model:airygrushin} is well-posed for $r=0$. In other words, suppose \eqref{eq:grushin}, then we have
\begin{equation}
\label{es:inv0}
\|u\|_{B_{z,\lambda,0}}+|u_-|\leqslant C(\|v\|_{L^2}+\langle\lambda-\Re z\rangle^{1/4}|v_0|+\langle\lambda-\Re z\rangle|v_+|).
\end{equation}
where $C$ is independent of $\lambda,z$.
\end{lem}
\begin{proof}
We first observe that for $1\leqslant j\leqslant N$,
\begin{equation*}
|\eta_j|\leqslant C\langle\lambda-\Re z\rangle
\end{equation*}
while for $j\geqslant N+1$
\begin{equation*}
|\eta_j|\geqslant C^{-1}(\langle\lambda-\Re z\rangle+\zeta_j').
\end{equation*}
The first inequality just follows the definition $\eta_j=e^{-2\pi i/3}\zeta_j'+\lambda-z$ and the assumption $|\Im z|<C_1$. When $\langle\lambda-\Re z\rangle\geqslant C\zeta_j'$, we can get the second inequality simply by estimating the real part using $|\Re\eta_j|\geqslant|\lambda-z|-C\zeta_j'$. Otherwise we use the imaginary part $\Im\eta_j=-(\sin2\pi/3)\zeta_j'-\Im z$ which does not vanish from the assumption on $N$. Therefore $|\Im\eta_j|\geqslant C^{-1}\zeta_j'$ and we also get the second inequality.

From the last equation in \eqref{eq:solution}, we easily get
\begin{equation}
|u_-|\leqslant C(\|v\|_{L^2}+|v_0|+\langle\lambda-\Re z\rangle|v_+|).
\end{equation}
To estimate $u$, we first write its orthogonal expansion in $L^2$ following the first three equations in \eqref{eq:solution}
\begin{equation*}
\begin{split}
u=&\;u_0f+\sum_{j=1}^\infty u_je_j\\
=&\;v_0\left(f-\sum_{j=1}^\infty f_je_j\right)
+\sum_{j=1}^Nv_+(j)e_j
+\sum_{j=N+1}^\infty\eta_j^{-1}(v_j-e^{-2\pi i/3}e_j(0)v_0)e_j\\
=&\;\sum_{j=1}^Nv_+(j)e_j
+\sum_{j=N+1}^\infty\eta_j^{-1}(v_j-e^{-2\pi i/3}e_j(0)v_0)e_j
\end{split}
\end{equation*}
which shows that
\begin{equation*}
\begin{split}
\|u\|_{L^2}^2=&\;\sum_{j=1}^N|v_+(j)|^2
+\sum_{j=N+1}^\infty|\eta_j|^{-2}|v_j-e^{-2\pi i/3}e_j(0)v_0|^2\\
\leqslant&\;C|v_+|^2+C\langle\lambda-\Re z\rangle^{-2}\|v\|_{L^2}^2
+C|v_0|^2\sum_{j=N+1}^\infty|\eta_j|^{-2}|e_j(0)|^2.
\end{split}
\end{equation*}

To treat the last term, we need a careful study of Airy functions. Recall that
\begin{equation*}
e_j(0)=\Ai(-\zeta_j')/\|\Ai\|_{L^2(-\zeta_j',\infty)}.
\end{equation*}
From the asymptotics \eqref{asy:Airy-neg}, it is not difficult to see that
\begin{equation*}
\zeta_j'=(\frac{3}{2}j\pi)^{2/3}(1+o(1)),\;\;\;j\to\infty
\end{equation*}
and
\begin{equation*}
\Ai(-\zeta_j')=(-1)^{j-1}\pi^{-1/2}(\frac{3}{2}j\pi)^{-1/6}(1+o(1)),\;\;\;
j\to\infty.
\end{equation*}
To compute the normalizing factor, we use
\begin{equation*}
\begin{split}
\|\Ai\|_{L^2(-\zeta_{k+1}',-\zeta_k')}^2\;&= (1+o(1))\pi^{-1}\int_{\zeta_k'}^{\zeta_{k+1}'}t^{-1/2}
|\sin(\frac{2}{3}t^{3/2}+\frac{\pi}{4})|^2dt\\
\;&=(1+o(1))\pi^{-1}
\int_{\frac{2}{3}\zeta_k'^{3/2}}^{\frac{2}{3}\zeta_{k+1}'^{3/2}}
|\sin(s+\frac{\pi}{4})|^2s^{-2/3}ds=\frac{1}{2}(k\pi)^{-2/3}(1+o(1)),
\end{split}
\end{equation*}
as $k\to\infty$. Here in the second step, we use the natural change of variables $s=\frac{2}{3}t^{3/2}$ while in the third step, we use that $s=\frac{2}{3}\zeta_k'^{3/2}(1+o(1))=k\pi(1+o(1))$ on $(\frac{2}{3}\zeta_k'^{3/2},\frac{2}{3}\zeta_{k+1}'^{3/2})$ and the integral of $|\sin(s+\frac{\pi}{4})|^2$ over this interval is equal to
\begin{equation*}
(1+o(1))\int_{k\pi}^{(k+1)\pi}|\sin(s+\frac{\pi}{4})|^2ds\
=\frac{\pi}{2}(1+o(1)).
\end{equation*}
Therefore
\begin{equation*}
\begin{split}
\|\Ai\|_{L^2(-\zeta_j',\infty)}^2=&\;
\|\Ai\|_{L^2(-\zeta_1',\infty)}^2+\sum_{k=1}^{j-1}
\|\Ai\|_{L^2(-\zeta_{k+1}',-\zeta_k')}^2\\
=&\;c_0(1+o(1))\sum_{k=1}^{j-1}k^{-2/3}=c_0j^{1/3}(1+o(1)).
\end{split}
\end{equation*}
As a consequence, we have
\begin{equation*}
|e_j(0)|^2=c_1j^{-2/3}(1+o(1)),\;\;\;j\to\infty
\end{equation*}
for some constant $c_1>0$. Now we can compute
\begin{equation*}
\begin{split}
\sum_{j=N+1}^\infty|\eta_j|^{-2}|e_j(0)|^2
\leqslant&\; C\sum_{j=N+1}^\infty j^{-2/3}(\langle\lambda-\Re z\rangle+\zeta_j')^{-2}\\
\leqslant&\; C\sum_{j=N+1}^\infty j^{-2/3}(\langle\lambda-\Re z\rangle+j^{2/3})^{-2}\\
\leqslant&\; C\int_1^\infty s^{-2/3}(\langle\lambda-\Re z\rangle+s^{2/3})^{-2}ds\\
\leqslant&\; C\langle\lambda-\Re z\rangle^{-3/2}\int_0^\infty
t^{-2/3}(1+t^{2/3})^{-2}dt\leqslant C\langle\lambda-\Re z\rangle^{-3/2},
\end{split}
\end{equation*}
where the last step we use the change of variable $s=\langle\lambda-\Re z\rangle^{3/2}t$. This gives the following estimate on the $L^2$-norm of $u$:
\begin{equation}
\label{es:L2norm}
\langle\lambda-\Re z\rangle\|u\|_{L^2}\leqslant C(\|v\|_{L^2}+\langle\lambda-\Re z\rangle^{1/4}|v_0|+\langle\lambda-\Re z\rangle|v_+|).
\end{equation}
Now since
\begin{equation*}
(D_t^2+t)u=e^{2\pi i/3}(v-R_-^0u_--(\lambda-z)u),
\end{equation*}
we have
\begin{equation*}
\|(D_t^2+t)u\|_{L^2}\leqslant
C(\|v\|_{L^2}+|u_-|+\langle\lambda-\Re z\rangle\|u\|_{L^2})
\end{equation*}
Now we can use a variation of \eqref{eq:normairy}
\begin{equation*}
\|u\|_{B_{z,\lambda,0}}\leqslant C(\|(D_t^2+t)u\|_{L^2}+\langle\lambda-\Re z\rangle\|u\|_{L^2})
\end{equation*}
and \eqref{es:L2norm} to get \eqref{es:inv0}.
\end{proof}

The next step is to consider adding a small exponential weight, i.e. $r\in(0,r_0)$ for $r_0$ small.
\begin{lem}
\label{lem:airy2}
There exists $r_0>0$ such that the Grushin problem \eqref{model:airygrushin} is uniformly well-posed for $r\in(0,r_0)$. More precisely, suppose \eqref{eq:grushin}, then we have
\begin{equation}
\label{es:invr}
\|u\|_{B_{z,\lambda,r}}+|u_-|\leqslant C(\|v\|_{L^2_r}+\langle\lambda-\Re z\rangle^{1/4}|v_0|+\langle\lambda-\Re z\rangle|v_+|).
\end{equation}
where $C$ is independent of $\lambda,z$ and $r$.
\end{lem}
\begin{proof}
We introduce
\begin{equation*}
\begin{split}
\mathcal{P}_\lambda^r(z)=&\;
\left(
  \begin{array}{ccc}
    e^{rt/2} & 0 & 0 \\
    0 & 1 & 0 \\
    0 & 0 & 1 \\
  \end{array}
\right)\mathcal{P}_0^\lambda
\left(
  \begin{array}{cc}
    e^{-rt/2} & 0 \\
    0 & 1 \\
  \end{array}
\right)\\
=&\;\mathcal{P}_\lambda(z)+
\left(
  \begin{array}{cc}
    e^{rt/2}(P_\lambda-z)e^{-rt/2} & (e^{rt/2}-1)R_-^0 \\
    \gamma_1(e^{-rt/2}-1) & 0 \\
    R_+^0(e^{-rt/2}-1) & 0 \\
  \end{array}
\right)
\end{split}
\end{equation*}

By the interpolation estimate \eqref{es:interpolation}, we have \begin{equation*}
D_t=O(\langle\lambda-\Re z\rangle^{-1/2}):
B_{z,\lambda,0}\to L^2,
\end{equation*}
thus
\begin{equation*}
e^{rt/2}(P_\lambda-z)e^{-rt/2}=e^{-2\pi i/3}(irD_t-\frac{1}{4}r^2)=O(r\langle\lambda-\Re z\rangle^{-1/2}):
B_{z,\lambda,0}\to L^2.
\end{equation*}

Next, by \eqref{es:boundary},
\begin{equation*}
\gamma_0=O(\langle\lambda-\Re z\rangle^{-3/4})
:B_{z,\lambda,0}\to\mathbb{C},
\end{equation*}
so
\begin{equation*}
\gamma_1(e^{-rt/2}-1)=-\frac{r}{2}\gamma_0=O(r\langle\lambda-\Re z\rangle^{-1/2})
:B_{z,\lambda,0}\to\mathbb{C}_{\langle\lambda-\Re z\rangle^{1/4}}.
\end{equation*}

Also by the super exponential decay of $e_j, j=1,\ldots,N$: $\|(e^{-rt/2}-1)e_j(t)\|_{L^2}=o(1)$, so
\begin{equation*}
R_+^0(e^{-rt/2}-1)=o(1):B_{z,\lambda,0}\to
\mathbb{C}^N_{\langle\lambda-\Re z\rangle}.
\end{equation*}
Similarly, we have $\|(e^{rt/2}-1)e_j(t)\|_{L^2}=o(1)$, and
\begin{equation*}
(e^{rt/2}-1)R_-^0=o(1):\mathbb{C}^N\to L^2.
\end{equation*}

We see that $\mathcal{P}_\lambda^r(z)$ is a small perturbation of $\mathcal{P}_\lambda(z)$ in the sense that
\begin{equation*}
\mathcal{P}_\lambda^r(z)-\mathcal{P}_\lambda(z)=o(1):
\mathcal{B}_{z,\lambda,0}\to\mathcal{H}_{z,\lambda,0}
\end{equation*}
uniformly in $z,\lambda$ as $r\to0+$.
Therefore
\begin{equation*}
\mathcal{P}_\lambda^r(z):
\mathcal{B}_{z,\lambda,0}\to\mathcal{H}_{z,\lambda,0}
\end{equation*}
is uniformly invertible when $r\in[0,r_0]$ for some small $r_0>0$.
Now we note that
\begin{equation*}
\|u\|_{B_{z,\lambda,r}}\sim\|e^{rt/2}u\|_{B_{z,\lambda,0}}
\end{equation*}
uniformly in $z,\lambda$ and $r\in[0,r_0]$ which again follows from the interpolation estimate \eqref{es:interpolation} for $D_t$. This finishes the proof of the lemma.
\end{proof}

In particular, from \eqref{eq:solution}, we see that the inverse of $\mathcal{P}_\lambda(z)$ is given by
\begin{equation*}
\mathcal{E}_\lambda(z)=\left(
                         \begin{array}{ccc}
                           E & K & E_+ \\
                           E_- & K_- & E_{-+} \\
                         \end{array}
                       \right)
:\mathcal{H}_{z,\lambda,r}\to\mathcal{B}_{z,\lambda,r},
\end{equation*}
where
\begin{equation}
\label{model:airye-+}
E_{-+}\in\hom(\mathbb{C}^N,\mathbb{C}^N),\;\;
(E_{-+})_{1\leqslant j,k\leqslant n}=-\eta_j\delta_{ij}.
\end{equation}

\subsection{Dependence on parameters}
Now we shall modify our Grushin problem so that we get nice global symbolic properties. For $0<\delta\ll1$, we put
\begin{equation*}
e_j^{\lambda,\delta}(t)=\Lambda^{1/2}e_j(\Lambda t), \Lambda=\langle\delta\lambda\rangle^{1/2}
\end{equation*}
which also forms an orthonormal basis for $L^2([0,\infty))$. We notice that
\begin{equation*}
\partial_\lambda^k\Lambda=O_k(1)\delta^k\Lambda^{1-2k},\;\;\;
\|\partial_\lambda^ke_j^{\lambda,\delta}\|_{L^2}
=O_k(1)\delta^k\Lambda^{-2k}.
\end{equation*}
In particular,
\begin{equation*}
\|e_j^{\lambda,\delta}-e_j\|_{L^2}\leqslant C\delta|\lambda|.
\end{equation*}
We define $R_+^{\lambda,\delta}$ and $R_-^{\lambda,\delta}$ by replacing $e_j$ with $e_j^{\lambda,\delta}$ in the definition of $R_+^0$ and $R_-^0$, we obtain
\begin{equation}
\label{model:modairy}
\mathcal{P}_\lambda^\delta(z)=\left(
                         \begin{array}{cc}
                           P_\lambda-z & R_-^{\lambda,\delta} \\
                           \gamma_1 & 0 \\
                           R_+^{\lambda,\delta} & 0 \\
                         \end{array}
                       \right):
\mathcal{B}_{z,\lambda,r}\to\mathcal{H}_{z,\lambda,r}
\end{equation}
and
\begin{equation*}
\mathcal{P}_\lambda^\delta(z)-\mathcal{P}_\lambda(z)=\left(
                         \begin{array}{cc}
                           0 & O(|\lambda|\delta) \\
                           0 & 0 \\
                           O(|\lambda|\delta) & 0 \\
                         \end{array}
                       \right):
\mathcal{B}_{z,\lambda,r}\to\mathcal{H}_{z,\lambda,r}
\end{equation*}
Thus for $|\lambda|\delta\ll1$ we get the uniform invertibility of $\mathcal{P}_\lambda^\delta(z)$. To get the same estimate for all $\lambda$, we need to assume
\begin{equation}
|\Re z|\ll\frac{1}{\delta},
\end{equation}
so that $|\lambda|\gg1+|\Re z|$ and we have the invertibility of
$\left(
   \begin{array}{c}
     P_\lambda-z \\
     \gamma_1 \\
   \end{array}
 \right)$ without the correcting terms $R_\pm^{\lambda,\delta}$.
Notice that in such situation $\langle\lambda\rangle\sim\langle\lambda-\Re z\rangle$ with a $\delta$-dependent constant. All our estimates will depend on $\delta$.

\begin{lem}
\label{lem:airyinv}
For $|\lambda|\gg1+|\Re z|$ and $|\Im z|<C_1$, there exists a constant $C>0$ independent of $z$ and $\lambda$ such that for any $u\in B_{z,\lambda,0}$,
\begin{equation}\label{eq:largelambdainv1}
|\langle(P_\lambda-z)u,u\rangle|+\langle\lambda-\Re z\rangle^{-1/2}|\gamma_1u|^2\geqslant C^{-1}\langle\lambda-\Re z\rangle\|u\|_{L^2}^2.
\end{equation}
Furthermore, for small $r$,
\begin{equation}\label{eq:largelambdainv2}
\left(
   \begin{array}{c}
     P_\lambda-z \\
     \gamma_1 \\
   \end{array}
 \right)u=
 \left(
   \begin{array}{c}
     v \\
     v_0 \\
   \end{array}
 \right)\;\;\Rightarrow\;\;
\|u\|_{B_{z,\lambda,r}}\leqslant C(\|v\|_{L^2_r}+\langle\lambda-\Re z\rangle^{1/4}|v_0|).
\end{equation}
\end{lem}
\begin{proof}
It is possible to repeat the argument as in Lemma \ref{lem:airy1} using orthogonal expansion with respect to $(e_j)$. We present here another proof by using the Poisson operator $K_\lambda:\mathbb{C}\to B_{z,\lambda,0}$, satisfying
\begin{equation*}
P_\lambda K_\lambda=0,\;\;\; \gamma_1K_\lambda=\Id.
\end{equation*}
This Poisson operator is given by multiplying $f=f_\lambda$ which is the solution to the equation
\begin{equation*}
e^{-2\pi i/3}(D_t^2+t)f+\lambda f=0,\;\;\; f'(0)=1.
\end{equation*}
We can give an explicit expression of $f$ in terms of the Airy function:
\begin{equation*}
f_\lambda(t)=\Ai'(e^{2\pi i/3}\lambda)^{-1}\Ai(t+e^{2\pi i/3}\lambda).
\end{equation*}
Notice that all the zeroes of $\Ai$ and $\Ai'$ lie on the negative real axis, this expression is well-defined as $\lambda$ is real.

We shall apply the asymptotic formulas for Airy function and its derivatives \eqref{asy:Airy-com} to study the $L^2$-norm of $f_\lambda$. First we consider the case $\lambda>0$, then
\begin{equation*}
\Ai'(e^{2\pi i/3}\lambda)=-(2\sqrt{\pi})^{-1}e^{\pi i/6}e^{\lambda^{3/2}}\lambda^{1/4}(1+O(\lambda^{-3/2})).
\end{equation*}
and
\begin{equation*}
\Ai(t+e^{2\pi i/3}\lambda)
=(2\sqrt{\pi})^{-1}e^{-\zeta}z^{-1/4}(1+O(|\zeta|^{-1}))
\end{equation*}
where
\begin{equation*}
z=t+e^{2\pi i/3}\lambda,\;\;\; |z|=(t^2-t\lambda+\lambda^2)^{1/2},\;\;\; \zeta=\frac{2}{3}z^{3/2}.
\end{equation*}
We change variables by letting $\arg z=\frac{\pi}{2}-\theta$, then $\theta\in[-\frac{\pi}{6},\frac{\pi}{2})$ and
\begin{equation*}
t=\frac{\lambda}{2}+\frac{\sqrt{3}}{2}\lambda\tan\theta,\;\;\; |z|=\frac{\sqrt{3}}{2}\lambda\sec\theta,\;\;\;
\zeta=\frac{\sqrt{3}}{4}\lambda^{3/2}e^{i(3\pi/4-3\theta/2)}
\sec^{3/2}\theta.
\end{equation*}
We have the following uniform asymptotic formulas in $\lambda$ and $\theta$ for $f_\lambda(t)$:
\begin{equation*}
f_\lambda(t)=g(\lambda)e^{\lambda^{3/2}\psi(\theta)}
e^{-i(7\pi/24-\theta/4)}(\sec^{-1/4}\theta)
(1+O(\lambda^{-3/2}\sec^{-3/2}\theta)).
\end{equation*}
where
\begin{equation*}
g(\lambda)=(\sqrt{3}/2)^{-1/4}\lambda^{-1/2}(1+O(\lambda^{-3/2})),\;\;\;
\psi(\theta)=-\frac{2}{3}-\frac{\sqrt{3}}{4}e^{i(3\pi/4-3\theta/2)}
\sec^{3/2}\theta.
\end{equation*}
Therefore
\begin{equation*}
\|f_\lambda\|_{L^2(0,\infty)}^2
=\frac{\sqrt{3}}{2}\lambda|g(\lambda)|^2
\int_{-\pi/6}^{\pi/2}
e^{\lambda^{3/2}\varphi(\theta)}
(\sec^{3/2}\theta)(1+O(\lambda^{-3/2}\sec^{-3/2}\theta))d\theta,
\end{equation*}
where
\begin{equation*}
\varphi(\theta)=2\Re\psi(\theta)=2\left[-\frac{2}{3}
-\frac{\sqrt{3}}{4}\sec^{3/2}\theta
\cos(\frac{3\pi}{4}-\frac{3\theta}{2})\right]
\end{equation*}
satisfies
\begin{equation*}
\varphi(-\pi/6)=0,\;\;\; \lim_{\theta\to\pi/2-0}\varphi(\theta)=-\infty,
\end{equation*}
and
\begin{equation*}
\varphi'(\theta)=-\frac{3\sqrt{3}}{4}\sec^{5/2}\theta
\sin(\frac{3\pi}{4}-\frac{\theta}{2})<-\frac{3\sqrt{3}}{8}<0, \;\;\; \theta\in[-\frac{\pi}{6},\frac{\pi}{2}).
\end{equation*}
Therefore integration by part gives us
\begin{equation}
\|f_\lambda\|=O(\lambda^{-3/4}).
\end{equation}
Now for every $u\in B_{z,\lambda,0}$, let $v=u-K_\lambda(\gamma_1u)=u-u'(0)f_\lambda$, we have $v'(0)=0$. Now we can write
\begin{equation*}
\begin{split}
\langle(P_\lambda-z)u,u\rangle=\langle(P_\lambda-z)v,v\rangle+
\overline{\gamma_1u}\langle(P_\lambda-z)v,f_\lambda\rangle\\
-z(\gamma_1u)\langle
f_\lambda,v\rangle-z|u'(0)|^2\|f_\lambda\|^2_{L^2}.
\end{split}
\end{equation*}
For the second term on the right-hand side, we integrate by parts:
\begin{equation*}
\begin{split}
\langle(P_\lambda-z)v,f_\lambda\rangle=&\;-e^{-2\pi i/3}v(0)
+\langle v,(P_\lambda-z)^\ast f_\lambda\rangle\\
=&\;-e^{-2\pi i/3}v(0)+(\lambda(1-e^{2\pi i/3})-\bar{z})\langle v,f_\lambda\rangle.
\end{split}
\end{equation*}
Therefore
\begin{equation*}
\begin{split}
\langle(P_\lambda-z)u,u\rangle=&\;e^{-2\pi i/3}\langle(D_t^2+t)v,v\rangle+(\lambda-z)\|v\|^2
-e^{-2\pi i/3}(\overline{\gamma_1u})v(0)\\
&\;+\overline{\gamma_1u}(\lambda(1-e^{2\pi i/3})-\bar{z})\langle v,f_\lambda\rangle-z(\gamma_1u)\langle f_\lambda,v\rangle-z|\gamma_1u|^2\|f_\lambda\|^2_{L^2},
\end{split}
\end{equation*}
where we notice that $\langle(D_t^2+t)v,v\rangle$ is always nonnegative. This gives
\begin{equation*}
\begin{split}
|\langle(P_\lambda-z)u,u\rangle|\geqslant&\;\Re(e^{\pi i/3}\langle(P_\lambda-z)u,u\rangle\\
\geqslant&\;\frac{1}{2}\langle(D_t^2+t)v,v\rangle
+C^{-1}\langle\lambda-\Re z\rangle\|v\|^2-\epsilon\langle\lambda-z\rangle^{1/2}|v(0)|^2\\
&\;-\epsilon\langle\lambda-\Re z\rangle\|v\|^2
-O_\epsilon(\langle\lambda-z\rangle^{-1/2})|\gamma_1u|^2
\end{split}
\end{equation*}
Now by choosing $\epsilon$ small enough but fixed and using
\begin{equation*}
\langle\lambda-z\rangle^{1/2}|v(0)|^2\leqslant
2\langle\lambda-z\rangle^{1/2}\|D_tv\|\|v\|
\leqslant\|D_tv\|^2+\langle\lambda-\Re z\rangle\|v\|^2
\end{equation*}
and $\langle(D_t^2+t)v,v\rangle\geqslant\|D_tv\|^2$ to deduce that
\begin{equation*}
|\langle(P_\lambda-z)u,u\rangle|\geqslant
C^{-1}\langle\lambda-\Re z\rangle\|v\|^2-C\langle\lambda-\Re z\rangle^{-1/2}|\gamma_1u|^2
\end{equation*}
by $\|u\|^2\leqslant C(\|v\|^2+\langle\lambda-\Re z\rangle^{-3/2}|\gamma_1u|^2)$, we can conclude the proof of \eqref{eq:largelambdainv1} for $\lambda>0$. For $\lambda<0$, we can get similarly $\|f_\lambda\|=O(|\lambda|^{-3/4})$ and then use
\begin{equation*}
|\langle(P_\lambda-z)u,u\rangle|\geqslant
\Re(-\langle(P_\lambda-z)u,u\rangle)
\end{equation*}
to reproduce the argument above and prove \eqref{eq:largelambdainv1}.

Now we prove \eqref{eq:largelambdainv2}. For $r=0$, we can see from
\eqref{eq:largelambdainv1},
\begin{equation*}
\|u\|^2_{L^2}\leqslant C\langle\lambda-\Re z\rangle^{-1}\|(P_\lambda-z)u\|_{L^2}\|u\|_{L^2}+C\langle\lambda-\Re z\rangle^{-3/2}|\gamma_1u|^2.
\end{equation*}
Therefore
\begin{equation*}
\|u\|_{L^2}\leqslant C\langle\lambda-\Re z\rangle^{-1}\|(P_\lambda-z)u\|+C\langle\lambda-\Re z\rangle^{-3/4}|\gamma_1u|^2
\end{equation*}
which proves \eqref{eq:largelambdainv2} for $r=0$. For small $r$, we can simply repeat the conjugation and perturbation argument as in the \ref{lem:airy1} to conclude the uniform invertibility.
\end{proof}

Now we give the desired invertibility for the full operator in the Grushin problem.
\begin{prop}
\label{prop:airy}
For $|\lambda|\geqslant1/(C\delta)$ and $|\Re z|\ll1/\delta$, $r\in[0,r_0]$ with $r_0>0$ small enough,
\begin{equation}
\label{es:grushin}
\mathcal{P}_\lambda^\delta\left(
                            \begin{array}{c}
                              u \\
                              u_- \\
                            \end{array}
                          \right)=
                          \left(
                            \begin{array}{c}
                              v \\
                              v_0 \\
                              v_+ \\
                            \end{array}
                          \right)\;\;\Rightarrow\;\;
\left\|\left(
         \begin{array}{c}
           u \\
           u_- \\
         \end{array}
       \right)
\right\|_{\mathcal{B}_{z,\lambda,r}}\leqslant C
\left\|\left(
         \begin{array}{c}
           v \\
           v_0 \\
           v_+
         \end{array}
       \right)
\right\|_{\mathcal{H}_{z,\lambda,r}}.
\end{equation}
Moreover, we have the following mapping properties of $\mathcal{P}_\lambda^\delta(z)$ and its inverse $\mathcal{E}_\lambda^\delta(z)$:
\begin{equation}
\label{pro:mapping}
\begin{split}
\|\partial_\lambda^k\mathcal{P}_\lambda^\delta(z)
\|_{\mathcal{L}(\mathcal{B}_{z,\lambda,r},\mathcal{H}_{z,\lambda,r})}
\leqslant&\; C_k\langle\lambda-\Re z\rangle^{-k},\\\|\partial_\lambda^k\mathcal{E}_\lambda^\delta(z)
\|_{\mathcal{L}(\mathcal{H}_{z,\lambda,r},\mathcal{B}_{z,\lambda,r})}
\leqslant&\; C_k\langle\lambda-\Re z\rangle^{-k}.
\end{split}
\end{equation}
\end{prop}
\begin{proof}
Again, we start with $r=0$. Let
\begin{equation*}
\Pi=R_-R_+:L^2\to(\ker R_+)^\perp=\Image R_-=\bigoplus\limits_{j=1}^N\mathbb{C}e_j^{\lambda,\delta}
\end{equation*}
be the orthogonal projection. Then since
\begin{equation*}
\|D_t^2e_j^{\lambda,\delta}\|_{L^2}=O(\langle\delta\lambda\rangle),\;\;\;
\|te_j^{\lambda,\delta}\|_{L^2}=O(\langle\delta\lambda\rangle^{-1/2}),
\end{equation*}
we have $\|(P_\lambda-z)|_{\Image R_-}\|=O(\langle\lambda-\Re z\rangle)$. Also it is easy to see $\|R_+\|=\|R_-\|=1$. Since $\Pi u=R_-R_+u=R_-v_+$, we have
\begin{equation*}
\|\Pi u\|_{L^2}\leqslant|v_+|
\end{equation*}
and
\begin{equation}
\label{es:piu}
\|(P_\lambda-z)\Pi u\|_{L^2}\leqslant O(\langle\lambda-\Re z\rangle)|v_+|.
\end{equation}
On the other hand, by the previous lemma,
\begin{equation*}
\begin{split}
\|(I-\Pi)u\|^2_{L^2}\leqslant&\;C\langle\lambda-\Re z\rangle^{-1}|\langle(P_\lambda-z)(I-\Pi)u,(I-\Pi)u\rangle|\\
&\;\;\;\;+C\langle\lambda-\Re z\rangle^{-3/2}|\gamma_1(I-\Pi)u|^2\\
\end{split}
\end{equation*}
For the first term, we have
\begin{equation*}
\begin{split}
\langle(P_\lambda-z)(I-\Pi)u,(I-\Pi)u\rangle=&\;
\langle(I-\Pi)(P_\lambda-z)(I-\Pi)u,u\rangle\\
=&\;\langle(I-\Pi)(P_\lambda-z)u,u\rangle-
\langle(I-\Pi)(P_\lambda-z)\Pi u,u\rangle\\
=&\;\langle(I-\Pi)(v-R_-u_-),u\rangle-
\langle(P_\lambda-z)\Pi u,(I-\Pi)u\rangle\\
=&\;\langle(I-\Pi)v,u\rangle-
\langle(P_\lambda-z)\Pi u,(I-\Pi)u\rangle\\
=&\;\langle v,(I-\Pi)u\rangle-
\langle(P_\lambda-z)\Pi u,(I-\Pi)u\rangle.\\
\end{split}
\end{equation*}
For the second term, we use $\gamma_1\Pi=0$ to get
\begin{equation*}
\gamma_1(I-\Pi)u=\gamma_1u=v_0.
\end{equation*}
Therefore
\begin{equation*}
\begin{split}
\|(I-\Pi)u\|_{L^2}^2
\leqslant&\;C\langle\lambda-\Re z\rangle^{-1}(\|v\|_{L^2}+\|(P_\lambda-z)\Pi u\|_{L^2})\|(I-\Pi)u\|\\
&\;\;\;\;+C\langle\lambda-\Re z\rangle^{-3/2}|v_0|\\
\end{split}
\end{equation*}
and thus
\begin{equation}
\label{es:ipiu}
\begin{split}
\|(I-\Pi)u\|_{L^2}
\leqslant&\;C\langle\lambda-\Re z\rangle^{-1}(\|v\|_{L^2}+\|(P_\lambda-z)\Pi u\|_{L^2})
+C\langle\lambda-\Re z\rangle^{-3/4}|v_0|\\
\leqslant&\;C\langle\lambda-\Re z\rangle^{-1}\|v\|_{L^2}+|v_+|
+C\langle\lambda-\Re z\rangle^{-3/4}|v_0|.
\end{split}
\end{equation}
Combining \eqref{es:piu} and \eqref{es:ipiu}, we have
\begin{equation*}
\langle\lambda-\Re z\rangle\|u\|_{L^2}\leqslant C(\|v\|_{L^2_r}
+\langle\lambda-\Re z\rangle^{1/4}|v_0|
+\langle\lambda-\Re z\rangle|v_+|.).
\end{equation*}
Since
\begin{equation*}
u_-=R_+R_-u_-=R_+(v-(P_\lambda-z)u)=R_+v-R_+(P_\lambda-z)u,
\end{equation*}
we have
\begin{equation*}
|u_-|\leqslant\|v\|_{L^2}+\|R_+(P_\lambda-z)u\|_{L^2}
\leqslant\|v\|_{L^2}+C\sum_{j=1}^N
|\langle(P_\lambda-z)u,e_j^{\lambda,\delta}\rangle|.
\end{equation*}
To estimate the sum, we integrate by parts and get
\begin{equation*}
\langle(P_\lambda-z)u,e_j^{\lambda,\delta}\rangle
=\langle u,(P_\lambda-z)^\ast e_j^{\lambda,\delta}\rangle
+e^{-2\pi i/3}u'(0)e_j^{\lambda,\delta}(0).
\end{equation*}
where $(P_\lambda-z)^\ast=e^{2\pi i/3}(D_t^2+t)+\lambda-\bar{z}$ is the formal adjoint of $P_\lambda-z$ so
\begin{equation*}
\|(P_\lambda-z)^\ast e_j^{\lambda,\delta}\|_{L^2}=O(\langle\lambda-\Re z\rangle).
\end{equation*}
In addition, we have $u'(0)=v_0$ and by definition of $e_j^{\lambda,\delta}$,
\begin{equation*}
e_j^{\lambda,\delta}(0)=O(\langle\delta\lambda\rangle^{1/4}),
\end{equation*}
which shows that
\begin{equation*}
|\langle(P_\lambda-z)u,e_j^{\lambda,\delta}\rangle|
\leqslant C\langle\lambda-\Re z\rangle\|u\|+C\langle\lambda-\Re z\rangle^{1/4}|v_0|.
\end{equation*}
As a consequence,
\begin{equation*}
|u_-|\leqslant C(\|v\|_{L^2_r}
+\langle\lambda-\Re z\rangle^{1/4}|v_0|
+\langle\lambda-\Re z\rangle|v_+|).
\end{equation*}
Now as in Lemma \ref{lem:airy1}, we can use the equation $(P_\lambda-z)u=v-R_-u_-$ to give the estimates on the $L^2$ norm of $D_t^2u$ and $tu$. This finishes the proof of \eqref{es:grushin} for $r=0$.

To extend this to $r\in[0,r_0]$ for some small $r_0>0$, we notice that \begin{equation*}
\|(e^{\pm rt/2}-1)e_j^{\lambda,\delta}\|=\|(e^{\pm r\langle\delta\lambda\rangle^{-1/2}t/2}-1)e_j\|=o(1)
\end{equation*}
uniformly as $r\to0$ which allow us to repeat the argument in Lemma \ref{lem:airy2}.

Finally, since for $k>1$,
\begin{equation*}
\partial_\lambda^k\mathcal{P}_\lambda^\delta(z)=
\left(
  \begin{array}{cc}
    \delta_{1k} & \partial_\lambda^kR^{\lambda,\delta}_+ \\
    0 & 0 \\
    \partial_\lambda^kR^{\lambda,\delta}_- & 0 \\
  \end{array}
\right)
\end{equation*}
and
\begin{equation*}
\|\partial_\lambda^ke_j^{\lambda,\delta}\|_{L^2_r}
=O_k(1)\delta^k\langle\delta\lambda\rangle^{-k}
=O_k(1)\langle\lambda-\Re z\rangle^{-k},
\end{equation*}
we get the mapping properties of $\mathcal{P}_\lambda^\delta(z)$ in \eqref{pro:mapping}. For its inverse $\mathcal{E}_\lambda^\delta(z)$, \eqref{es:grushin} gives the mapping property when $k=0$. The case $k>0$ follows directly from the case $k=0$ and the Leibnitz rule.
\end{proof}

To end this part, we study the $(-+)$-component of $\mathcal{E}_\lambda^\delta$:
\begin{prop}
\label{prop:airy:e-+}
For any $\epsilon>0$, $|\lambda|\leqslant1/(C\sqrt{\delta}),|\Re z|\ll1/\sqrt{\delta}$ sufficiently small depending on $\epsilon$,
\begin{equation}
\label{pro:e-+:perturb}
\|E_{-+}^\delta(z,\lambda)-\diag(z-\lambda-e^{-2\pi i/3}\zeta_j')\|
\leqslant\epsilon
\end{equation}
and if $\det E_{-+}^\delta(z,\lambda)=0$, then
\begin{equation}
\label{pro:e-+:zero}
z=\lambda+e^{-2\pi i/3}\zeta_j'.
\end{equation}
Moreover, for $|\lambda|\gg1+|\Re z|$,
\begin{equation}
\label{pro:e-+:inverse}
\|E_{-+}^\delta(z,\lambda)^{-1}
\|_{\mathcal{L}(\mathbb{C}^N,\mathbb{C}^N)}
=O(\langle\lambda-\Re z\rangle^{-1})
\end{equation}
\end{prop}

\begin{proof}
The \eqref{pro:e-+:perturb} follows from the perturbation
\begin{equation*}
\|E_{-+}^\delta(z,\lambda)-\diag(z-\lambda-e^{-2\pi i/3}\zeta_j')\|
\leqslant O(\lambda|\delta|)\langle\lambda-\Re z\rangle.
\end{equation*}
Let us recall the general fact, (which is essentially the Schur complement formula, see e.g. \cite{HS} or \cite{SZ7} in the setting of Grushin problems),
\begin{equation*}
(E^{\delta}_{-+})^{-1}=-R_+^{\lambda,\delta}
\left(
  \begin{array}{c}
    P_\lambda-z \\
    \gamma_1 \\
  \end{array}
\right)^{-1}
\left(
  \begin{array}{c}
    R_-^{\lambda,\delta} \\
    0 \\
  \end{array}
\right).
\end{equation*}
Since $\left(
  \begin{array}{c}
    P_\lambda-z \\
    \gamma_1 \\
  \end{array}
\right)$ is not invertible precisely when $\eta_j=e^{-2\pi i/3}\zeta_j'+\lambda-z=0$, (in which case $e_j$ is in the kernel), the same is true for $E_{-+}^{\delta}$. This gives \eqref{pro:e-+:zero}. Finally, in the case $|\lambda|\gg1+|\Re z|$, by \ref{lem:airyinv},
$\left(
  \begin{array}{c}
    P_\lambda-z \\
    \gamma_1 \\
  \end{array}
\right)$ is invertible. Therefore $E_{-+}^\delta:\mathbb{C}^N_{\langle\lambda-\Re z\rangle}\to\mathbb{C}^N$ is also invertible, which gives \eqref{pro:e-+:inverse}.
\end{proof}

\subsection{The ``easy'' model}
When $|\lambda|\gg1+|\Re z|$ and $|\Im z|<C_1$, we can consider an even simpler model problem with the operator \eqref{model:easy} which is already invertible. To obtain the uniform symbolic properties, we shall construct the Grushin problem using the same correction terms $R_\pm^{\lambda,\delta}$ as in \eqref{model:modairy}. We define
\begin{equation}
\mathcal{P}_\lambda^\#(z)=\left(
                            \begin{array}{cc}
                              P_\lambda^\#-z & R_-^{\lambda,\delta}\\
                              \gamma_1 & 0 \\
                              R_+^{\lambda,\delta} & 0 \\
                            \end{array}
                          \right):\mathcal{B}_{\lambda,r}^\#
                          \to\mathcal{H}_{\lambda,r}^\#,
\end{equation}
where the spaces $\mathcal{B}_{\lambda,r}^\#$ and $\mathcal{H}_{\lambda,r}^\#$  are defined by
\begin{equation}
\begin{split}
\mathcal{B}_{\lambda,r}^\#&\;=B_{\lambda,r}^\#\times\mathbb{C}^N,
B_{\lambda,r}^\#=\{u\in L^2_r:D_t^2u\in L^2_r\},\\
\left\|\left(
         \begin{array}{c}
           u \\
           u_- \\
         \end{array}
       \right)
\right\|_{\mathcal{B}_{\lambda,r}^\#}&\;
=\langle\lambda\rangle\|u\|_{L^2_r}+\|D_t^2u\|_{L^2_r}+|u_-|,\\
\mathcal{H}_{\lambda,r}^\#&\;
=L^2_r\times\mathbb{C}_{\langle\lambda\rangle^{1/4}}\times
\mathbb{C}^N_{\langle\lambda\rangle},\\
\left\|\left(
         \begin{array}{c}
           v \\
           v_0 \\
           v_+
         \end{array}
       \right)
\right\|_{\mathcal{H}_{\lambda,r}^\#}&\;=\|v\|_{L^2_r}
+\langle\lambda\rangle^{1/4}|v_0|
+\langle\lambda\rangle|v_+|.
\end{split}
\end{equation}

\begin{prop}
\label{prop:easy}
For $|\lambda|\gg1+|\Re z|$, and $r\in[0,r_0]$ with $r_0>0$ small enough, $\mathcal{P}_\lambda^\#(z):\mathcal{B}_{\lambda,r}^\#
\to\mathcal{H}_{\lambda,r}^\#$ is uniformly invertible. We have the mapping properties for $\mathcal{P}_\lambda^\#(z)$ and its inverse $\mathcal{E}_\lambda^\#(z)$:
\begin{equation}
\begin{split}
\|\partial_\lambda^k\mathcal{P}_\lambda^\#(z)
\|_{\mathcal{L}(\mathcal{B}_{\lambda,r}^\#,
\mathcal{H}_{\lambda,r}^\#)}
\leqslant&\; C_k\langle\lambda\rangle^{-k}\\
\|\partial_\lambda^k\mathcal{E}_\lambda^\#(z)
\|_{\mathcal{L}(\mathcal{H}_{\lambda,r}^\#,
\mathcal{B}_{\lambda,r}^\#)}
\leqslant&\; C_k\langle\lambda\rangle^{-k}.
\end{split}
\end{equation}
Moreover, the $(-+)$-component of $\mathcal{E}_\lambda^\#$ satisfies:
\begin{equation}
E_{-+}^\#(z,\lambda)^{-1}=O(\langle\lambda\rangle^{-1}).
\end{equation}
\end{prop}
\begin{proof}
The proof is almost identical to the Airy model problem we discussed above. To make the argument work, we only need to replace the Poisson operator $K_\lambda$ by $K_\lambda^\#$ satisfying
\begin{equation*}
P_\lambda^\#K_\lambda^\#=0, \gamma_1K_\lambda^\#=0,
\end{equation*}
which is given by multiplying the function
\begin{equation*}
f_\lambda^\#=-e^{\pi i/3}\lambda^{-1/2}\exp(-e^{-\pi i/3}\lambda^{1/2}t).
\end{equation*}
When $\lambda$ is negative, we choose the branch $\lambda^{1/2}=i(-\lambda)^{1/2}$ so $f_\lambda^\#$ has exponential decay. An easy calculation shows that
\begin{equation*}
\|f_\lambda\|_{L^2}=O(|\lambda|^{-3/4}),
\end{equation*}
and therefore all our arguments in Lemma \ref{lem:airyinv}, thus in Proposition \ref{prop:airy} and \ref{prop:airy:e-+} can be carried out in the same way. We shall omit the details here.
\end{proof}

\subsection{The $\mu$-dependent construction.}
Now we shall put the parameter $\mu$ back into the operator and describe the necessary modification we need to make in the model problem. The idea is to change coordinates $t=\mu^{-1/3}\tilde{t}$ in \eqref{model:airy} which will reduce to the case $\mu=1$. From our discussion, it will be clear that when $\mu$ varies in a compact subset of $(0,\infty)$ all the estimates will be uniformly in $\mu$ provided that we construct all the operators accordingly and replace the the eigenvalues $\zeta_j'$ of Neumann Airy operator $D_t^2+t$ by $\mu^{2/3}\zeta_j'$. More precisely, we have the following Grushin problem
\begin{equation}
\label{model:modairymu}
\mathcal{P}_\lambda^\delta(z)=\left(
                         \begin{array}{cc}
                           P_\lambda-z & R_-^{\lambda,\delta,\mu} \\
                           \gamma_1 & 0 \\
                           R_+^{\lambda,\delta,\mu} & 0 \\
                         \end{array}
                       \right):
\mathcal{B}_{z,\lambda,r}\to\mathcal{H}_{z,\lambda,r}
\end{equation}
where the spaces $\mathcal{B}_{z,\lambda,r}, \mathcal{H}_{z,\lambda,r}$ are as before and we reintroduce the additional parameter $\mu$ in the operators
\begin{equation*}
\begin{split}
P_\lambda-z=&\; e^{-2\pi i/3}(D_t^2+\mu t)+\lambda-z\\
R_+^{\lambda,\delta,\mu}u=&\;(\langle u,e_{j,\mu}^{\lambda,\delta}\rangle)_{1\leqslant j\leqslant N}\\
R_-^{\lambda,\delta,\mu}u_-=&\;\sum_{j=1}^Nu_-(j)e_{j,\mu}^{\lambda,\delta}
\end{split}
\end{equation*}
with
\begin{equation}
\label{def:mueigens}
e_{j,\mu}^{\lambda,\delta}(t)=\mu^{1/6}e_j^{\lambda,\delta}(\mu^{1/3}t)
=\mu^{1/6}\langle\delta\lambda\rangle^{1/4}
e_j(\mu^{1/3}\langle\delta\lambda\rangle^{1/2}t).
\end{equation}
In the mean time, we also replace the $R_\pm^{\lambda,\delta}$ in the easy model by $R_\pm^{\lambda,\delta,\mu}$. Then all the previous results hold uniformly in $\mu\in[C^{-1},C]\subset(0,\infty)$ with possibly a smaller $r_0>0$ due to the change of variable $t=\mu^{-1/3}\tilde{t}$.

\section{Second microlocal symbol class for Grushin problems}
\label{sec:second}
In this part, we consider the symbol class for the operator \eqref{op:comb} near the boundary where we have the expression in coordinates $(t=h^{-2/3}x_n,x')$,
\begin{equation}
\begin{split}
P-z=&\;e^{-2\pi i/3}(D_t^2+2tQ(h^{2/3}t,x',hD_{x'};h))\\
&+h^{-2/3}(R(x',hD_{x'};h)-w)+F(h^{2/3}t,x')h^{2/3}D_t-z,
\end{split}
\end{equation}
and $\gamma=\gamma_1+h^{2/3}k\gamma_0$. The difficulty is that though this operator has a good symbol property, out construction of the inverse requires a symbol class that has a non-classical behavior. More precisely, the symbol class will contain functions of $h^{-2/3}(R(x',\xi')-w)$ and near the glancing hypersurface $\Sigma_w=\{R(x',\xi')=w\}$. We lose $2/3$-power of $h$ each time we differentiate such symbols in the transversal direction. Symbol classes characterizing such non-classical behavior are introduced in \cite{SZ6} and we shall follow their approach.

\subsection{Second microlocalization with respect to a hypersurface}
In this part, we review some facts about the second microlocalization with respect to a hypersurface. For details, see \cite{SZ6}.

We always assume that $X$ is a $n$-dimensional compact smooth manifold and $\Sigma\subset T^\ast X$ is a smooth compact hypersurface. In our application, $X=\partial\mathcal{O}$ will be the boundary of the obstacle and $\Sigma=\Sigma_w=\{(x',\xi')\in T^\ast\partial\mathcal{O}:R(x',\xi')=w\}$ will be the glancing hypersurface. We shall also fix a distance function $d(\Sigma,\cdot)$ on $T^\ast X$ as the absolute value of a defining function of $\Sigma$. In particular, $d(\Sigma,\cdot)$ vanishes only on $\Sigma$ and behaves like $\langle\xi\rangle$ near the infinity in $T^\ast X$.

To start with, we recall the standard class of semiclassical symbols on $T^\ast X$, see e.g. \cite{DS}, \cite{Ma} and \cite{Z2},
\begin{equation*}
S^{m,k}(T^\ast X)=\{a\in C^\infty(T^\ast X\times(0,1]):|\partial_x^\alpha\partial_\xi^\beta a(x,\xi;h)|\leqslant C_{\alpha\beta}h^{-m}\langle\xi\rangle^{k-|\beta|}\}.
\end{equation*}
One can also study the more general class $S^{m,k}_\delta$ with $0\leqslant\delta<\frac{1}{2}$ where the right-hand side is replaced by $C_{\alpha\beta}h^{-m-\delta(|\alpha|+|\beta|)}
\langle\xi\rangle^{k-(1-\delta)|\beta|+\delta|\alpha|}$.

Now for any $0\leqslant\delta<1$ we define a class of symbols associated to $\Sigma$: $a\in S_{\Sigma,\delta}^{m,k_1,k_2}(T^\ast X)$ if
\begin{equation}\label{2msymbol}
\begin{split}
&\text{near } \Sigma: V_1\cdots V_{l_1}W_1\cdots W_{l_2}a=O(h^{-m-\delta l_1}
\langle h^{-\delta} d(\Sigma,\cdot)\rangle^{k_1}),\\
&\text{where } V_1, \ldots, V_{l_1} \text{ are vector fields tangent to } \Sigma,\\
&\text{ and }W_1, \ldots, W_{l_2} \text{ are any vector fields};\\
&\text{away from } \Sigma: \partial_x^\alpha\partial_\xi^\beta a(x,\xi;h)=O(h^{-m-\delta k_1}\langle\xi\rangle^{k_2-|\beta|}).
\end{split}
\end{equation}

To define the corresponding class of operators $\Psi_{\Sigma,\delta}^{m,k_1,k_2}$, we start locally by assuming $\Sigma$ is of the normal form $\Sigma_0=\{\xi_1=0\}$. Then near $\xi_1=0$, we can write $a=a(x,\xi,\lambda;h)$ with $\lambda=h^{-\delta}\xi_1$. Then \eqref{2msymbol} becomes
\begin{equation}
\partial_x^\alpha\partial_\xi^\beta\partial_\lambda^l
a(x,\xi,\lambda,h)=O(h^{-m})\langle\lambda\rangle^{k-l},
\end{equation}
which we shall write $a=\tilde{O}(h^{-m}\langle\lambda\rangle^k)$. Then we can define
\begin{equation}
\widetilde{\Op}_h(a)u(x)=\frac{1}{(2\pi h)^n}\int e^{\frac{i}{h}\langle x-y,\xi\rangle}a(x,\xi,h^{-\delta}\xi_1,h)u(y)dyd\xi.
\end{equation}
Then as in the standard semiclassical calculus, we have the composition formula: for $a=\tilde{O}(h^{-m_1}\langle\lambda\rangle^{k_1})$ and $b=\tilde{O}(h^{-m_2}\langle\lambda\rangle^{k_2})$,
\begin{equation*}
\widetilde{\Op}_h(a)\circ\widetilde{\Op}_h(b)
=\widetilde{\Op}_h(a\#_hb) \mod{\Psi^{-\infty,-\infty}(X)},
\end{equation*}
where
\begin{equation*}
a\#_hb(x,\xi,\lambda;h)=\sum_{\alpha\in\mathbb{N}^n}
\frac{1}{\alpha!}(h\partial_{\xi'})^{\alpha'}
(h\partial_{\xi_1}+h^{1-\delta}\partial_\lambda)^{\alpha_1}a
D_x^\alpha b\in\tilde{O}(h^{-m_1-m_2}\langle\lambda\rangle^{k_1+k_2}).
\end{equation*}
We also have a version of Beals's characterization of pseudodifferential operators: Let $A=A_h:\mathcal{S}(\mathbb{R}^n)\to\mathcal{S}'(\mathbb{R}^n)$ and put $x'=(x_2,\ldots,x_n)$. Then $A=\tilde{\Op}_h(a)$ for some $a=\tilde{O}(h^{-m}\langle\lambda\rangle^k)$ if and only if for all $N,p,q\geqslant0$ and every sequence $l_j(x',\xi'),j=1,\ldots,N$ of linear forms on $\mathbb{R}^{2(n-1)}$ there exists $C>0$ such that
\begin{equation*}
\begin{split}
\|\ad_{l_1(x',hD_{x'})}\circ\cdots\circ\ad_{l_N(x',hD_{x'})}\circ
(\ad_{h^{1-\delta}D_{x_1}})^p\circ(\ad_{x_1})^q&Au\|_{(q-\min(k,0))}\\
\leqslant&\; Ch^{N+(1-\delta)(p+q)}\|u\|_{(\max(k,0))},
\end{split}
\end{equation*}
where $\|u\|_{(p)}=\|u\|_{L^2}+\|(h^{1-\delta}D_{x_1})^pu\|_{L^2}$.

The global definition of the class $\Psi_{\Sigma,\delta}^{m,k_1,k_2}(X)$ relies on the invariance of $\widetilde{\Op}_h(\tilde{O}(\langle\lambda\rangle^m))$ under conjugation by $h$-Fourier integral operators whose associated canonical relation fixed $\{\xi_1=0\}$. See Proposition 4.2 in \cite{SZ6}. Now we define $A\in\Psi_{\Sigma,\delta}^{m,k_1,k_2}(X)$ if and only if\\
(1) for any $m_0\in\Sigma$ and any $h$-Fourier integral operator $U:C^\infty(X)\to C^\infty(\mathbb{R}^n)$ elliptic near $((0,0),m_0)$ whose corresponding canonical transformation $\kappa$ satisfies $\kappa(m_0)=(0,0)$, $\kappa(\Sigma\cap V)\subset\{\xi_1=0\}$ for some neighborhood $V$ of $m_0$, we have
$UAU^{-1}=\widetilde{\Op}_h
(\tilde{O}(h^{-m}\langle\lambda\rangle^{k_1})$, microlocally near $(0,0)$;\\
(2) for any $m_0$ outside any fixed neighborhood of $\Sigma$, $A\in\Psi^{m+\delta k_1,k_2}(X)$ microlocally near $m_0$ in both classical and semiclassical sense.

In particular, we have the quantization map
\begin{equation*}
\Op_{\Sigma,h}:S_{\Sigma,\delta}^{m,k_1,k_2}(T^\ast X)\to\Psi_{\Sigma,\delta}^{m,k_1,k_2}(X),
\end{equation*}
and the principal symbol map
\begin{equation*}
\sigma_{\Sigma,h}:\Psi_{\Sigma,\delta}^{m,k_1,k_2}(X)\to
S_{\Sigma,\delta}^{m,k_1,k_2}(T^\ast X)/
S_{\Sigma,\delta}^{m-1+\delta,k_1-1,k_2-1}(T^\ast X).
\end{equation*}

For $a\in S^{m,k_1,-\infty}_{\Sigma,\delta}$ we introduce a notion of essential support. We say for an $h$-dependent family of sets $V_h\subset T^\ast X$,
\begin{equation*}
\esssupp a\cap V_h=\emptyset
\end{equation*}
if and only if there exists $\chi\geqslant0$, $\chi\in S^{0,0,-\infty}(T^\ast X)$, such that
\begin{equation*}
\chi|_{V_h}\geqslant1, \chi a\in S^{-\infty,-\infty}(T^\ast X).
\end{equation*}

As the standard case, if $a,b\in S_{\Sigma,\delta}^{m,k,-\infty}(T^\ast X)$ satisfies $\Op_{\Sigma,h}(a)=\Op_{\Sigma,h}(b)$, then
$\esssupp a=\esssupp b$. Therefore we can define for $A\in \Psi^{m,k,-\infty}_{\Sigma,\delta}(X)$ the semiclassical wave front set as $\WF_h(A)=\esssupp a$ if $A=\Op_{\Sigma,h}(a)$.

Now we generalize the symbol class to an arbitrary order function $m$ and vector valued as operators from a Banach space $\mathcal{B}$ to another Banach space $\mathcal{H}$. We assume that $m=m(x,\xi,\lambda;h)$ is an order function with respect to the metric $g=dx^2+d\xi^2/\langle\xi\rangle+d\lambda^2/\langle\lambda\rangle$ in the sense that
\begin{equation*}
|g_{(x,\xi,\lambda)}(y,\eta,\mu)|\leqslant c\Rightarrow
C^{-1}m(x,\xi,\lambda)\leqslant m(x+y,\xi+\eta,\lambda+\mu)
\leqslant Cm(x,\xi,\lambda).
\end{equation*}
(See \cite{Ho} for instance.) We also assume that $\mathcal{B}$ and $\mathcal{H}$ are equipped with $(x,\xi,\lambda;h)$-dependent norms $\|\cdot\|_{m_\mathcal{B}}$, $\|\cdot\|_{m_\mathcal{H}}$ which are equivalent to some fixed norm (may not uniformly), respectively. In addition, we assume that
the norms are continuous with respect to the metric $g$, uniformly with respect to $h$. Then we say that $a\in S_{\Sigma,\delta}(T^\ast X,m,\mathcal{L}(\mathcal{B},\mathcal{H}))$ if
\begin{equation}
\|a(x,\xi;h)u\|_{m_\mathcal{H}(x,\xi,\lambda;h)}\leqslant Cm(x,\xi,\lambda;h)\|u\|_{m_\mathcal{B}(x,\xi,\lambda;h)}, \lambda=h^{-\delta}d(\Sigma,\cdot), \text{ for all } u\in\mathcal{B},
\end{equation}
and if this statement is stable under applications of vector fields in the sense of \eqref{2msymbol}, namely,
\begin{equation}
\begin{split}
&\text{near } \Sigma: V_1\cdots V_{l_1}W_1\cdots W_{l_2}a=O_{\mathcal{L}(\mathcal{B},\mathcal{H})}(mh^{-\delta l_1}),\\
&\text{where } V_1, \ldots, V_{l_1} \text{ are vector fields tangent to } \Sigma,\\
&\text{ and }W_1, \ldots, W_{l_2} \text{ are any vector fields};\\
&\text{away from } \Sigma: \partial_x^\alpha\partial_\xi^\beta a(x,\xi;h)=O_{\mathcal{L}(\mathcal{B},\mathcal{H})}
(m\langle\xi\rangle^{-|\beta|}).
\end{split}
\end{equation}
Then we can obtain a class of operators $\Psi_{\Sigma,\delta}(X;m,\mathcal{L}(\mathcal{B},\mathcal{H}))$ and the corresponding principal symbol map
\begin{equation}
\begin{split}
\sigma_{\Sigma,h}:
\Psi_{\Sigma,\delta}&(X;m,\mathcal{L}(\mathcal{B},\mathcal{H}))\\
&\to S_{\Sigma,\delta}(T^\ast X;m,\mathcal{L}(\mathcal{B},\mathcal{H}))/
S_{\Sigma,\delta}(T^\ast X;m\langle h^{-\delta}d(\Sigma,\cdot)\rangle^{-1},
\mathcal{L}(\mathcal{B},\mathcal{H})).
\end{split}
\end{equation}

\subsection{Analysis near the glancing hypersurface}
We can use $|R(x',\xi')-w|$ as our distance function to the glancing hypersurface $\Sigma_w$ for which we shall perform the second microlocalization. First, we work near the glancing hypersurface, i.e. $|R(x',\xi')-w|\leqslant2C^{-1}$, then
$$\lambda=h^{-2/3}(R(x',\xi')-w)=O(h^{-2/3}).$$
We shall think of this as perturbation of the principal symbol
\begin{equation}
\left(
  \begin{array}{c}
    P_0-z \\
    \gamma_1 \\
  \end{array}
\right)=
\left(
  \begin{array}{c}
    e^{-2\pi i/3}(D_t^2+\mu t)+\lambda-z \\
    \gamma_1 \\
  \end{array}
\right),
\end{equation}
where
$$\mu=2Q(x',\xi')\in[C^{-1},C].$$
As in the previous section, we set up the Grushin problem by letting $R_\pm=R_\pm^{\lambda,\delta}$ there. Then we have the operator-valued symbol
\begin{equation}
\mathcal{P}_0(z)=
\left(
  \begin{array}{cc}
    P_0-z & R_- \\
    \gamma_1 & 0 \\
    R_+ & 0 \\
  \end{array}
\right)
\end{equation}
which is uniformly invertible in $\mathcal{L}(\mathcal{B}_{z,\lambda,r},\mathcal{H}_{z,\lambda,r})$ with inverse $\mathcal{E}_0(z)$.

For simplicity, let us pretend for now that $Q$ does not depend additionally in $h$, then by Taylor expansion with respect to $x_n=h^{2/3}t$, we have
$$\mathcal{P}(z)\equiv\mathcal{P}_0(z)+h^{2/3}
\mathcal{K}_0+\sum_{j=1}^\infty h^{2j/3}T^j\mathcal{P}_j+\sum_{j=1}^\infty h^{2j/3}T^{j-1}\mathcal{D}_j.$$
Here
$$\mathcal{K}_0=\left(
        \begin{array}{cc}
          0 & 0 \\
          k(x')\gamma_0 & 0 \\
          0 & 0 \\
        \end{array}
      \right),$$
$$\mathcal{P}_j=
\left(
  \begin{array}{cc}
    \frac{1}{j!}2e^{-2\pi i/3}t\partial_t^jQ(0,x',\xi') & 0 \\
    0 & 0 \\
    0 & 0 \\
  \end{array}
\right),$$
$$\mathcal{D}_j=
\left(
  \begin{array}{cc}
    \frac{1}{(j-1)!}\partial_t^{j-1}F(0,x')D_t & 0 \\
    0 & 0 \\
    0 & 0 \\
  \end{array}
\right),$$
and
$$T=\left(
      \begin{array}{ccc}
        t & 0 & 0\\
        0 & 0 & 0\\
        0 & 0 & 0\\
      \end{array}
    \right).$$

To find the inverse of such symbols, we shall take the approach similar to \S 1 of Sj\"{o}strand \cite{S1} which is again motivated by the work of Boutet de Monvel-Kree \cite{BK} on formal analytic symbols. Instead of considering a symbol $q=q(x,\xi;h)$, we deal with the formal operator
$$Q=q(x,\xi+hD_x;h)
\equiv\sum_{\alpha\in\mathbb{N}^{n-1}}\frac{1}{\alpha!}
\partial_\xi^\alpha q(x,\xi;h)(hD_x)^\alpha.$$
The symbol $q$ itself can be recovered by the formula
$$q=Q(1).$$
The advantage to work with this setting is that the composition formula
$$a\#_hb=\sum_{\alpha\in\mathbb{N}^{n-1}}
\frac{1}{\alpha!}(h\partial_\xi)^\alpha aD_x^\alpha b$$
becomes the formal composition of the corresponding formal operators $A$ and $B$:
$$a\#_hb=A\circ B(1).$$
Therefore to find the inverse of such a symbol is equivalent to find the inverse of the corresponding formal operator.

For this purpose, we shall introduce the following class of operators
\begin{equation*}
\mathfrak{A}=\sum_{k,\alpha}
(h^{2/3}T)^kA_{k,\alpha}(x',\xi,\lambda;h)D_{x'}^\alpha,
\end{equation*}
where
$$A_{k,\alpha}:\mathcal{B}_{z,\lambda,r}
\to\mathcal{H}_{z,\lambda,r}.$$
The inverse of such operators should be of the form
\begin{equation*}
\mathfrak{B}=\sum_{k,\alpha}
(h^{2/3}T)^kB_{k,\alpha}(x',\xi,\lambda;h)D_{x'}^\alpha,
\end{equation*}
where
$$B_{k,\alpha}:\mathcal{H}_{z,\lambda,r}\to\mathcal{B}_{z,\lambda,r}.$$
However, we should notice that the $T$ in the second class of operators should be interpreted as
$$T=\left(
      \begin{array}{ccc}
        t & 0\\
        0 & 0\\
      \end{array}
    \right)$$
acting on $\mathcal{B}_{z,\lambda,r}$ instead of on $\mathcal{H}_{z,\lambda,r}$. When needed, we shall write this one as $T_{\mathcal{B}}$ and the previous one as $T_{\mathcal{H}}$.

There are several technical issues about these two different operators $T$ that we have to deal with. First, $T$ is not a bounded operator on $\mathcal{B}_{z,\lambda,r}$ or $\mathcal{H}_{z,\lambda,r}$. We can deal with this issue by relaxing the exponentially weighted space.
$$T^k=O(1)C^kk^k(r-r')^{-k}
:\mathcal{B}_{z,\lambda,r}\to\mathcal{B}_{z,\lambda,r'}$$
if $r>r'$ and similar for $\mathcal{H}_{z,\lambda,r}\to\mathcal{H}_{z,\lambda,r'}$. Therefore we can work on the formal level and interpret the formal operators in the end as operators from $\mathcal{B}_{z,\lambda,r}$ to $\mathcal{H}_{z,\lambda,r'}$ (or similar operators with the weight function in the codomain relaxed to $r'$.)

The second issue comes from the non-commutativity of operators $T$ with $A_k$ or $B_k$. To compose two such operators $\mathfrak{A}$ and $\mathfrak{B}$, we are hoping to get a class of operators
\begin{equation*}
\mathfrak{C}=\sum_{k,\alpha}
(h^{2/3}T)^kC_{k,\alpha}(x',\xi',\lambda)D_{x'}^\alpha,
\end{equation*}
where
$$C_{k,\alpha}:\mathcal{H}_{z,\lambda,r}\to\mathcal{H}_{z,\lambda,r}
\text{ or }\mathcal{B}_{z,\lambda,r}\to\mathcal{B}_{z,\lambda,r},$$
depending on the order of composition. This composition will involve the ``commutators'' $\ad_T=[T,\cdot]$ which we should interpreted as
$$\ad_T(A)=T_{\mathcal{H}}A-AT_{\mathcal{B}},$$
$$\ad_T(B)=T_{\mathcal{B}}B-BT_{\mathcal{H}},$$
when it acts on different classes. We shall also need $\ad_T$ to act on the two different classes of $\mathfrak{C}$ and we shall interpret it accordingly.

This involves the study of stability of mapping properties of $A_k$ and $B_k$ under the ``commutator operation'' $\ad_T$. We first consider $\mathcal{P}_0$ to see its mapping properties and then adjust our definition of formal operators in a suitable way.

\begin{lem}
For $|\Re z|\ll1/\delta$, we have
\begin{equation}
\ad_T^k\mathcal{P}_0=O_k(\delta^{-k/2}\langle\lambda-\Re z\rangle^{-k/2})
:\mathcal{B}_{z,\lambda,r}\to\mathcal{H}_{z,\lambda,r}.
\end{equation}
\end{lem}
\begin{proof}
We have seen in the last section that this is true for $k=0$. A simple calculation gives
\begin{equation*}
\ad_T^k\mathcal{P}_0=
\left(
  \begin{array}{cc}
    \ad_t^k(P_0-z) & t^kR_- \\
    (-1)^k\gamma_1t^k & 0 \\
    (-1)^kR_+t^k & 0 \\
  \end{array}
\right),
\end{equation*}
where $\ad_t=[t,\cdot]$ is the commutator with multiplying $t$. For $k=1$,
$$\ad_t(P_0-z)=2ie^{-2\pi i/3}D_t=O(\langle\lambda-\Re z\rangle^{-1/2}):\mathcal{B}_{z,\lambda,r}\to L_r^2.$$
For $k=2$,
$$\ad_t^2(P_0-z)=-2e^{-2\pi i/3}=O(\langle\lambda-\Re z\rangle^{-1}):\mathcal{B}_{z,\lambda,r}\to L_r^2.$$
For $k>2$,
$$\ad_t^k(P_0-z)=0.$$
For $k=1$,
$$(-1)^k\gamma_1t^k=\gamma_0=O(\langle\lambda-\Re z\rangle^{-1/2}):
\mathcal{B}_{z,\lambda,r}\to \mathbb{C}_{\langle\lambda-\Re z\rangle^{1/4}}.$$
For $k>1$,
$$(-1)^k\gamma_1t^k=0.$$
Also for $k\geqslant1$, we have
$$R^+t^k=O_k(\delta^{-k/2}\langle\lambda-\Re z\rangle^{-k/2}):
\mathcal{B}_{z,\lambda,r}\to\mathbb{C}^N;$$
$$(-1)^kt^kR^-=O_k(\delta^{-k/2}\langle\lambda-\Re z\rangle^{-k/2}):
\mathbb{C}^N\to L^2_r.$$
Combining all these estimates together, we get the desired mapping properties for $\ad_T^k\mathcal{P}_0$.
\end{proof}

On the other hand, we also need the stability for $\mathcal{P}_0(z)$ under differentiation in $x',\xi',\lambda$ which will give the second microlocal symbol class which is simply
\begin{equation*}
\mathcal{P}_0(z)\in S_{\Sigma_w,2/3}(\partial\mathcal{O};
1,\mathcal{L}(\mathcal{B}_{z,\lambda,r},\mathcal{H}_{z,\lambda,r})).
\end{equation*}
We shall combine the two types of mapping properties together to get
\begin{equation*}
\partial_{x'}^\alpha\partial_{\xi'}^\beta
\partial_\lambda^l\ad_T^k\mathcal{P}_0(z)
=O(\delta^{-k/2}\langle\lambda-\Re z\rangle^{-l-k/2})
:\mathcal{B}_{z,\lambda,r}\to\mathcal{H}_{z,\lambda,r},
\end{equation*}
where the constants depending on $k,l,\alpha,\beta$. Now each of $\partial_{x'},\partial_{\xi'},\partial_\lambda$ and $\ad_T$ is a derivation provided if we interpret $\ad_T$ suitably. We get similar estimates for the inverse:
\begin{equation*}
\partial_{x'}^\alpha\partial_{\xi'}^\beta
\partial_\lambda^l\ad_T^k\mathcal{E}_0(z)
=O(\delta^{-k/2}\langle\lambda-\Re z\rangle^{-l-k/2})
:\mathcal{H}_{z,\lambda,r}\to\mathcal{B}_{z,\lambda,r},
\end{equation*}
since we have seen the estimates for $k=l=0,\alpha=\beta=0$ in last section. We can replace $\langle\lambda-\Re z\rangle$ by $\langle\lambda\rangle$ with the expense of $\delta$-dependent constants.

Also we have the symbol properties for $\mathcal{P}_j$, $\mathcal{D}_j$ and $K_0$:
\begin{equation*}
\partial_{x'}^\alpha\partial_{\xi'}^\beta
\partial_\lambda^l\ad_T^k\mathcal{P}_j(z)
=O(\langle\lambda\rangle^{-l-k/2})
:\mathcal{B}_{z,\lambda,r}\to\mathcal{H}_{z,\lambda,r},
\end{equation*}
\begin{equation*}
\partial_{x'}^\alpha\partial_{\xi'}^\beta
\partial_\lambda^l\ad_T^k\mathcal{D}_j(z)
=O(\langle\lambda\rangle^{-1/2-l-k/2})
:\mathcal{B}_{z,\lambda,r}\to\mathcal{H}_{z,\lambda,r},
\end{equation*}
and
\begin{equation*}
\partial_{x'}^\alpha\partial_{\xi'}^\beta
\partial_\lambda^l\ad_T^k\mathcal{K}_0(z)
=O(\langle\lambda\rangle^{-1/2-l-k/2})
:\mathcal{B}_{z,\lambda,r}\to\mathcal{H}_{z,\lambda,r}.
\end{equation*}
We remark that we neglect a number of simplifying features here, for example, for $\mathcal{K}_0$, only when all of $\beta$, $k$ and $l$ are zero, the operator does not vanish.

Now we can introduce the suitable class of formal operators:
\begin{equation}
\mathfrak{A}=\sum_{\alpha\in\mathbb{N}^{n-1},j,k,l,m\in\mathbb{N}}
(h^{2/3}T)^j(h^{2/3}\langle\lambda\rangle^{-1/2})^k
(h^{1/3}\langle\lambda\rangle^{-1})^lh^m
\mathcal{A}_{\alpha,j,k,l,m}(x',\xi',\lambda,z)D_{x'}^\alpha,
\end{equation}
with the mapping properties for $\mathcal{A}_{\alpha,j,k,l,m}$
\begin{equation}
\partial_{x'}^{\tilde{\alpha}}\partial_{\xi'}^{\tilde{\beta}}
\partial_\lambda^{\tilde{l}}\ad_T^{\tilde{k}}
\mathcal{A}_{\alpha,j,k,l,m}
=O(\langle\lambda\rangle^{-\tilde{l}-\tilde{k}/2})
:\mathcal{B}_{z,\lambda,r}\to\mathcal{H}_{z,\lambda,r}.
\end{equation}

We shall rewrite the operator $\mathcal{P}$ as
\begin{equation*}
\mathcal{P}(z)=h^{2/3}\mathcal{K}_0(x')+\sum_{j=0}^\infty h^{2j/3}T^j
(\mathcal{P}_j(x',\xi',\lambda,z;h)
+h^{2/3}\mathcal{D}_{j+1}(x';h)),
\end{equation*}
where $\mathcal{K}_0$ is as above and $\mathcal{P}_j$, $\mathcal{D}_j$ satisfies the same properties as above.

Then the associated formal operator $\mathfrak{P}$ is given by
\begin{equation*}
\begin{split}
\mathfrak{P}=&\;
\sum_{\alpha\in\mathbb{N}^{n-1}}
\frac{1}{\alpha!}\partial_{\xi'}^\alpha
(\mathcal{P}(x',\xi',\lambda,z;h))(hD_{x'})^\alpha\\
=&\;\sum_{\alpha\in\mathbb{N}^{n-1}}\frac{1}{\alpha!}
[\partial_{\xi''}^{\alpha''}
(\partial_{\xi_1}+h^{-2/3}\partial_\lambda)^{\alpha_1}\mathcal{P}]
(x',\xi',\lambda,z;h)(hD_{x'})^\alpha\\
=&\;\sum_{\alpha\in\mathbb{N}^{n-1}}\frac{1}{\alpha!}
[(h\partial_{\xi''})^{\alpha''}(h\partial_{\xi_1}
+h^{1/3}\partial_\lambda)^{\alpha_1}\mathcal{P}]
(x',\xi',\lambda,z;h)D_{x'}^\alpha\\
=&\;h^{2/3}\mathcal{K}_0+\sum_{j=1}^\infty h^{2j/3}T^{j-1}\mathcal{D}_j\\
&\;+
\sum_{\alpha\in\mathbb{N}^{n-1}}\frac{1}{\alpha!}
\sum_{j\in\mathbb{N}}h^{2j/3}T^j[(h\partial_{\xi''})^{\alpha''}
(h\partial_{\xi_1}+h^{1/3}\partial_\lambda)^{\alpha_1}\mathcal{P}_j]
(x',\xi',\lambda,z;h)D_{x'}^{\alpha'}
\end{split}
\end{equation*}
is in this class and with principal term $\mathcal{P}_0(x',\xi,\lambda,z)=\mathcal{P}_0(z)$. Here we write $\alpha'=(\alpha_1,\alpha'')$.

For the inverse, we introduce the class of operators $\mathfrak{B}$ of the same form as $\mathfrak{A}$ with $\mathcal{A}_{\alpha,j,k,l,m}$ replaced by $\mathcal{B}_{\alpha,j,k,l,m}$ satisfying
\begin{equation}
\partial_{x'}^{\tilde{\alpha}}\partial_{\xi'}^{\tilde{\beta}}
\partial_\lambda^{\tilde{l}}\ad_T^{\tilde{k}}
\mathcal{B}_{\alpha,j,k,l,m}
=O(\langle\lambda\rangle^{-\tilde{l}-\tilde{k}/2})
:\mathcal{H}_{z,\lambda,r}\to\mathcal{B}_{z,\lambda,r}.
\end{equation}

Then the composition of $\mathfrak{A}$ and $\mathfrak{B}$,
$$\mathfrak{C}=\mathfrak{A}\circ\mathfrak{B},\;\;\; (\text{or } \mathfrak{B}\circ\mathfrak{A}),$$
is of the same form as $\mathfrak{A}$ and $\mathfrak{B}$ with $\mathcal{A}_{\alpha,j,k,l,m}$ or $\mathcal{B}_{\alpha,j,k,l,m}$ replaced by $\mathcal{C}_{\alpha,j,k,l,m}$ satisfying
\begin{equation}
\partial_{x'}^{\tilde{\alpha}}\partial_{\xi'}^{\tilde{\beta}}
\partial_\lambda^{\tilde{l}}\ad_T^{\tilde{k}}
\mathcal{B}_{\alpha,j,k,l,m}
=O(\langle\lambda\rangle^{-\tilde{l}-\tilde{k}/2})
:\mathcal{H}_{z,\lambda,r}\to\mathcal{H}_{z,\lambda,r}.
\end{equation}
(or $\mathcal{B}_{z,\lambda,r}\to\mathcal{B}_{z,\lambda,r}$.)

Now the construction of the formal inverses is through the standard techniques of Neumann series.

\begin{lem}
If $\mathfrak{A}$ is as above with $\mathcal{A}_0$ invertible. Also $\mathcal{B}_0=\mathcal{A}_0^{-1}$ satisfying
$$\mathcal{B}_0=O(1):
\mathcal{H}_{z,\lambda,r}\to\mathcal{B}_{z,\lambda,r}.$$
Then there exists $\mathfrak{B}$ as above with the principal term $\mathcal{B}_0$ such that
$$\mathfrak{A}\circ\mathfrak{B}=\Id,\;\;\; \mathfrak{B}\circ\mathfrak{A}=\Id.$$
\end{lem}
\begin{proof}
Let $\mathfrak{C}=\mathfrak{A}\circ\mathfrak{B}_0$ where $\mathfrak{B}_0=\mathcal{B}_0$, then $\mathfrak{C}$ is as above with $\mathcal{C}_0=\mathcal{A}_0\circ\mathcal{B}_0=\Id$. Therefore we can form the formal Neumann series
$$\mathfrak{D}=\Id+(\Id-\mathfrak{C})
+(\Id-\mathfrak{C})\circ(\Id-\mathfrak{C})+\cdots$$
which again gives a formal operator as above. Then we can simply take $\mathfrak{B}=\mathcal{B}_0\circ\mathfrak{D}$ to get the right inverse. The left inverse can be constructed in the same way and the standard argument shows that the two must have the same formal expansions. And it is clear from the construction that the principal term of $\mathfrak{B}$ is $\mathcal{B}_0$.
\end{proof}

Now applying this lemma to $\mathfrak{P}$, we get an inverse $\mathfrak{E}$. Let $\mathcal{E}=\mathfrak{E}(1)$, we get a parametrix for $\mathcal{P}(z)$ in the region $|R(x',\xi')-w|\leqslant 2C^{-1}$:
\begin{equation}
\mathcal{E}(x',\xi',\lambda,z;h)=\sum_{j,k,l,m\in\mathbb{N}}
(h^{2/3}T)^j(h^{2/3}\langle\lambda\rangle^{-1/2})^k
(h^{1/3}\langle\lambda\rangle^{-1})^lh^m
\mathcal{E}_{0,j,k,l,m}(x',\xi',\lambda,z)
\end{equation}
with
\begin{equation}
\partial_{x'}^{\tilde{\alpha}}\partial_{\xi'}^{\tilde{\beta}}
\partial_\lambda^{\tilde{l}}\ad_T^{\tilde{k}}
\mathcal{E}_{0,j,k,l,m}
=O(\langle\lambda\rangle^{-\tilde{l}-\tilde{k}/2})
:\mathcal{H}_{z,\lambda,r}\to\mathcal{B}_{z,\lambda,r}.
\end{equation}

In particular, the principal term is exactly $\mathcal{E}_0$ as we constructed in the previous section.

\subsection{Analysis away from the glancing hypersurface}
Now we deal with the region $|R(x',\xi')-w|>C^{-1}$. In this case, $Q\ll|\lambda|=h^{-2/3}|R-w|$ so that we are working with the second model operator in last section where we regard $tQ(h^{2/3}t,x',\xi')$ also as a perturbation. Let
$$P_0^\#=e^{-2\pi i/3}D_t^2+\lambda,\;\;\; \lambda=h^{-2/3}(R(x',\xi')-w),$$
and $R_\pm$ as before. The operator-valued symbol
\begin{equation}
\mathcal{P}_0^\#(z)=
\left(
  \begin{array}{cc}
    P_0^\#-z & R_- \\
    \gamma_1 & 0 \\
    R_+ & 0 \\
  \end{array}
\right):\mathcal{B}_{\lambda,r}^\#\to\mathcal{H}_{\lambda,r}^\#
\end{equation}
is uniformly invertible with inverse $\mathcal{E}_0^\#(z)$ since $|\lambda|\geqslant h^{-2/3}/C\gg|\Re z|$. Moreover,
\begin{equation*}
\mathcal{P}_0^\#(z)\in S_{\Sigma_w,2/3}(\partial\mathcal{O};1,
\mathcal{L}(\mathcal{B}_{\lambda,r}^\#,\mathcal{H}_{\lambda,r}^\#)).
\end{equation*}

Recall the definition for the symbol class that away from the glancing hypersurface, the symbol behaves classically and we do not need to specify the derivative in $\lambda$. However, we need to consider the possibility that $\xi'$ may get large. More precisely, the symbol properties for $\mathcal{P}_0^\#$ and $\mathcal{E}_0^\#$ are given by
\begin{equation*}
\partial_{x'}^\alpha\partial_{\xi'}^\beta\ad_T^k\mathcal{P}_0^\#(z)
=O(\langle\xi'\rangle^{-|\beta|}\langle\lambda\rangle^{k/2})
:\mathcal{B}_{\lambda,r}^\#\to\mathcal{H}_{\lambda,r}^\#;
\end{equation*}
\begin{equation*}
\partial_{x'}^\alpha\partial_{\xi'}^\beta\ad_T^k\mathcal{P}_0^\#(z)
=O(\langle\xi'\rangle^{-|\beta|}\langle\lambda\rangle^{-k/2})
:\mathcal{H}_{\lambda,r}^\#\to\mathcal{B}_{\lambda,r}^\#;
\end{equation*}
where we notice that $|\lambda|^{-k/2}\sim(h^{-1/3}\langle\xi'\rangle)^{-k}$ and $Q(0,x',\xi')=O(h^{2/3})|\lambda|$. For the lower order term in the expansion
\begin{equation*}
\mathcal{P}(z)\equiv
h^{2/3}\mathcal{K}_0+
\sum_{j=0}^\infty(h^{2/3}T)^j\mathcal{P}_j^\#(x,\xi,z;h)
\end{equation*}
with $T$, $\mathcal{K}_0$ as before and
\begin{equation*}
\partial_{x'}^\alpha\partial_{\xi'}^\beta\ad_T^k\mathcal{P}_j^\#
=O(1)\langle\xi'\rangle^{-|\beta|}(h^{1/3}\langle\xi'\rangle^{-1})^k
:\mathcal{B}_{\lambda,r}^\#\to\mathcal{H}_{\lambda,r}^\#.
\end{equation*}

We proceed exactly as before to define the associated formal operator
\begin{equation*}
\mathfrak{P}^\#=\sum_{\alpha\in\mathbb{N}^{n-1}}\frac{1}{\alpha!}
((h\partial_{\xi'})^\alpha\mathcal{P})D_{x'}^\alpha.
\end{equation*}
This motivate us to consider the general class of formal operators of the form
\begin{equation}
\mathfrak{A}^\#=\sum_{\alpha\in\mathbb{N}^{n-1},j,k\in\mathbb{N}}
(h^{2/3}T)^j(h\langle\xi'\rangle)^k
\mathcal{A}^\#_{\alpha,j,k}(x',\xi',z;h)D_{x'}^\alpha
\end{equation}
with
\begin{equation}
\partial_{x'}^{\tilde{\alpha}}\partial_{\xi'}^{\tilde{\beta}}
\ad_T^{\tilde{k}}\mathcal{A}_{\alpha,j,k}
=O(1)\langle\xi'\rangle^{-|\tilde{\beta}|}
(h^{1/3}\langle\xi'\rangle^{-1})^{\tilde{k}}
:\mathcal{B}_{\lambda,r}^\#\to\mathcal{H}_{\lambda,r}^\#.
\end{equation}
So we see that $\mathfrak{P}^\#$ is in this class. The same argument as in the case near the glancing hypersurface shows that $\mathfrak{P}^\#$ has a formal inverse $\mathfrak{E}^\#$ of the same form satisfying the estimates with $\mathcal{H}^\#$ and $\mathcal{B}^\#$ exchanged. Therefore we have an inverse of $\mathcal{P}(z)$ in the region $|R(x',\xi')-w|\geqslant C^{-1}$,
\begin{equation}
\mathcal{E}^\#(x',\xi',z;h)=\mathfrak{E}^\#(1)=
\sum_{j,k\in\mathbb{N}}
(h^{2/3}T)^j(h\langle\xi'\rangle)^k
\mathcal{E}^\#_{j,k}(x',\xi',z;h)
\end{equation}
with the following mapping properties
\begin{equation}
\partial_{x'}^\alpha\partial_{\xi'}^\beta\ad_T^{\tilde{k}}
\mathcal{E}^\#_{j,k}=O(1)\langle\xi'\rangle^{-|\beta|}
(h^{1/3}\langle\xi'\rangle)^k
:\mathcal{H}^\#_{\lambda,r}\to\mathcal{B}^\#_{\lambda,r}.
\end{equation}

\subsection{Analysis in the intermediate region}
\label{sec:inter}
In the intermediate region $C^{-1}\leqslant|R(x',\xi')-w|\leqslant2C^{-1}$, we observe that both cases reduce to the simpler expansions that coincide with each other.
The key point is that in this region both $\lambda$ and $\xi'$ will be irrelevant. In fact, $|\xi'|$ is bounded and $\lambda\sim h^{-2/3}$. Therefore we have the expansions
\begin{equation*}
\mathcal{E}(x',\xi',z;h)
=\sum_{j,k\in\mathbb{N}}(h^{2/3}T)^jh^k
\mathcal{E}_{j,k}(x',\xi',z;h)
\end{equation*}
where
\begin{equation*}
\partial_{x'}^\alpha\partial_{\xi'}^\beta\ad_T^{\tilde{k}}
\mathcal{E}_{j,k}=O(h^{k/3})
:\mathcal{H}_{z,\lambda,r}\to\mathcal{B}_{z,\lambda,r};
\end{equation*}
and
\begin{equation*}
\mathcal{E}^\#(x',\xi',z;h)
=\sum_{j,k\in\mathbb{N}}(h^{2/3}T)^jh^k
\mathcal{E}^\#_{j,k}(x',\xi',z;h)
\end{equation*}
where
\begin{equation*}
\partial_{x'}^\alpha\partial_{\xi'}^\beta\ad_T^{\tilde{k}}
\mathcal{E}^\#_{j,k}=O(h^{k/3})
:\mathcal{H}_{\lambda,r}^\#\to\mathcal{B}_{\lambda,r}^\#.
\end{equation*}
Of course the same is true for $\mathcal{P}$ with $\mathcal{B}$ and $\mathcal{H}$ exchanged. Therefore if we introduce spaces $\mathcal{B}$ and $\mathcal{H}$ which agrees with $\mathcal{B}_{z,\lambda,r}$ and $\mathcal{H}_{z,\lambda,r}$ microlocally in $|R(x',\xi')-w|<2C^{-1}$, also agrees with $\mathcal{B}_{\lambda,r}^\#$ and $\mathcal{H}_{\lambda,r}^\#$ microlocally in $|R(x',\xi')-w|>C^{-1}$. Then this coincidence on the intermediate region shows that the symbol $\mathcal{P}$ and $\mathcal{E}$ satisfies the global construction at least near the boundary.

\section{Global Grushin problems}
\label{sec:global}

\subsection{Estimates away from the boundary}
We begin by recalling the following estimates away from the boundary. Let
\begin{equation}
D(\alpha)=\{x\in\mathbb{R}^n\setminus\mathcal{O}:
d(x,\partial\mathcal{O})>\alpha\},
\end{equation}
and $H_h^k(\Omega)$ be the semiclassical Sobolev space on an open set $\Omega\subset\mathbb{R}^n$ (or on a compact manifold which we shall set to be $\partial\mathcal{O}$ later). Then in \cite[section 7]{SZ6}, the following proposition is proved.

\begin{prop}
Let $0<\epsilon<\frac{2}{3}$, $|\Re z|\leqslant L$, $|\Im z|\leqslant C$, then there exists $h_0=h_0(L)$ such that for $0<h<h_0(L)$, there exists maps $E_\epsilon, K_\epsilon$ defined on $C^\infty_c(D(h^\epsilon))$, with the properties $(P-z)E_{\epsilon}=I+K_{\epsilon}$ and
\begin{equation}
\begin{split}
E_\epsilon=&\;O(h^{2/3-\epsilon}): L^2(D(h^\epsilon))\to H_h^2(\mathbb{R}^n\setminus\mathcal{O}),\\
K_\epsilon=&\;O(e^{-C^{-1}h^{-1+\frac{3\epsilon}{2}}}): L^2(D(h^{\epsilon}))\to H_h^k(\mathbb{R}^n\setminus\mathcal{O}),\;\; \forall k\in\mathbb{R}.\\
\end{split}
\end{equation}
Moreover, for any fixed $\gamma\in(0,1)$, we can construct $E_\epsilon$ and $K_\epsilon$ such that for $u\in C_c^\infty(D(h^\epsilon))$, $E_\epsilon u$ and $K_\epsilon u$ are supported in $D((1-\gamma)h^\epsilon)$.
\end{prop}

We remark that we can not use Neumann series and this proposition to give an inverse of $P-z$ since the support of $K_\epsilon u$ is larger than $u$ in general.

\subsection{Setting up for global Grushin problems}

To study the global Grushin problem, we introduce the spaces for $w\in W\Subset(0,\infty)$, $0<\delta\ll1$, $0\leqslant r\leqslant r_0$:
\begin{equation*}
\begin{split}
\mathcal{B}_{w,r,\delta}=&\;H^2(\mathbb{R}^n\setminus\mathcal{O})
\times L^2(\partial\mathcal{O};\mathbb{C}^N),\\
\mathcal{H}_{w,r}=&\;L^2(\mathbb{R}^n\setminus\mathcal{O})\times H^{1/2}(\partial\mathcal{O})\times H^2(\partial\mathcal{O};\mathbb{C}^N).
\end{split}
\end{equation*}
with the norms which coincides the ones introduced in the previous sections in each of the regions we considered. We need to translate the norms to $x_n$-coordinates by the relation $x_n=h^{2/3}t$.

Let
\begin{equation}
\begin{split}
\left\|\left(
         \begin{array}{c}
           u \\
           u_- \\
         \end{array}
       \right)
\right\|_{\mathcal{B}_{w,r,\delta}}=&\;
h^{-2/3}\|e^{r\psi(x_n)/2h^{2/3}}
(hD_{x_n})^2u\|_{L^2(\mathbb{R}^n\setminus\mathcal{O})}\\
&+h^{-2/3}\|e^{r\psi(x_n)/2h^{2/3}}\chi(x_n/\delta)
x_nu\|_{L^2(\mathbb{R}^n\setminus\mathcal{O})}\\
&+\|e^{r\psi(x_n)/2h^{2/3}}\langle x_n\rangle^{-2}\langle h^{-2/3}(-h^2\Delta_{\partial\mathcal{O}}-w)\rangle u\|_{L^2(\mathbb{R}^n\setminus\mathcal{O})}\\
&+h^{-2/3}\|e^{r\psi(x_n)/2h^{2/3}}(1-\chi(x_n/\delta))
u\|_{H_h^2(\mathbb{R}^n\setminus\mathcal{O})}\\
&+h^{1/3}\|u_-\|_{L^2(\partial\mathcal{O};\mathbb{C}^N)},\\
\left\|\left(
         \begin{array}{c}
           v \\
           v_0 \\
           v_+
         \end{array}
       \right)
\right\|_{\mathcal{H}_{w,r}}=&\;\|e^{r\psi(x_n)/2h^{2/3}}
v\|_{L^2(\mathbb{R}^n\setminus\mathcal{O})}\\
&+h^{1/3}\|\langle h^{-2/3}(-h^2\Delta_{\partial\mathcal{O}}-w)\rangle^{1/4}
v_0\|_{L^2(\partial\mathcal{O})}\\
&+h^{1/3}\|\langle h^{-2/3}(-h^2\Delta_{\partial\mathcal{O}}-w)\rangle v_+\|_{L^2(\partial\mathcal{O};\mathbb{C}^N)},
\end{split}
\end{equation}
where the weight function $\psi\in C^\infty([0,\infty);[0,1])$ satisfying $\psi(t)=t$ for $t<\frac{1}{2}$ and $\psi(t)=1$ for $t\geqslant1$; and the cut-off function $\chi\in C^\infty([0,\infty);[0,1])$ satisfying $\chi(t)=1$ for $t<1$ and $\chi(t)=0$ for $t>2$. Here we still use the geodesic normal coordinates $(x',x_n)\in\partial\mathcal{O}\times(0,\infty)$ for $\mathbb{R}^n\setminus\mathcal{O}$ as introduced before.

First we claim that
\begin{equation*}
\left(
  \begin{array}{cc}
    P-z & 0 \\
    \gamma_1 & 0 \\
    0 & 0 \\
  \end{array}
\right):\mathcal{B}_{w,r}\to\mathcal{H}_{w,r}.
\end{equation*}

In fact, we can decompose $u\in H^2(\mathbb{R}^n\setminus\mathcal{O})$ as $u=u_1+u_2$ where $\supp u_1\subset\{x_n\leqslant3\delta\} $ and $\supp u_2\subset\{x_n\geqslant2\delta\}$. Then we see that
\begin{equation*}
\left\|\left(
         \begin{array}{c}
           u \\
           0 \\
         \end{array}
       \right)
\right\|_{\mathcal{B}_{w,r,\delta}}
\sim\left\|\left(
         \begin{array}{c}
           u_1 \\
           0 \\
         \end{array}
       \right)
\right\|_{\mathcal{B}_{w,r,\delta}}+
\left\|\left(
         \begin{array}{c}
           u_2 \\
           0 \\
         \end{array}
       \right)
\right\|_{\mathcal{B}_{w,r,\delta}}.
\end{equation*}
We notice that
\begin{equation*}
\begin{split}
\left\|\left(
         \begin{array}{c}
           u_1 \\
           0 \\
         \end{array}
       \right)
\right\|_{\mathcal{B}_{w,r,\delta}}
\sim &\;h^{-2/3}\|e^{rx_n/2h^{2/3}}
(hD_{x_n})^2u_1\|_{L^2(\mathbb{R}^n\setminus\mathcal{O})}\\
&+h^{-2/3}\|e^{rx_n/2h^{2/3}}\chi(x_n/\delta)
x_nu\|_{L^2(\mathbb{R}^n\setminus\mathcal{O})}\\
&+\|e^{rx_n/2h^{2/3}}\langle h^{-2/3}(-h^2\Delta_{\partial\mathcal{O}}-w)\rangle u\|_{L^2(\mathbb{R}^n\setminus\mathcal{O})},
\end{split}
\end{equation*}
so the estimate
\begin{equation*}
\|e^{r\psi(x_n)/2h^{2/3}}(P-z)u_1\|_{L^2}\leqslant
\left\|\left(
         \begin{array}{c}
           u_1 \\
           0 \\
         \end{array}
       \right)
\right\|_{\mathcal{B}_{w,r,\delta}}
\end{equation*}
follows from the change of variable $x_n=h^{2/3}t$ and the result in \ref{sec:model} (only the boundedness of $P-z:B_{z,\lambda,r}\to L^2_r$). Also notice that
\begin{equation*}
\left\|\left(
         \begin{array}{c}
           u_2 \\
           0 \\
         \end{array}
       \right)
\right\|_{\mathcal{B}_{w,r,\delta}}
\sim h^{-2/3}\|e^{r\psi(x_n)/2h^{2/3}}u_2
\|_{H^2_h(\mathbb{R}^n\setminus\mathcal{O})},
\end{equation*}
so we can easily deduce that
\begin{equation*}
\|e^{r\psi(x_n)/2h^{2/3}}(P-z)u_2\|_{L^2}\leqslant
\left\|\left(
         \begin{array}{c}
           u_2 \\
           0 \\
         \end{array}
       \right)
\right\|_{\mathcal{B}_{w,r,\delta}}.
\end{equation*}
Finally we need to estimate $\gamma u$. We shall use the fact that
\begin{equation*}
\gamma_0=O(h^{-1/2}):H_h^2(\mathbb{R}^n\setminus\mathcal{O})\to
H_h^{3/2}(\partial\mathcal{O})
\end{equation*}
and
\begin{equation*}
h\gamma_1=O(h^{-1/2}):H_h^2(\mathbb{R}^n\setminus\mathcal{O})\to
H_h^{1/2}(\partial\mathcal{O})
\end{equation*}
which follows from the estimates of non-semiclassical restriction operators. Therefore we have
\begin{equation*}
\begin{split}
&\;h^{1/3}\|\langle h^{-2/3}(-h^2\Delta_{\partial\mathcal{O}}-w)\rangle^{1/4}
(\gamma u)\|_{L^2(\partial\mathcal{O})}\\
\leqslant&\; h^{1/6}\|\gamma u\|_{H_h^{1/2}(\partial\mathcal{O})}
\leqslant h^{1/6}\|h^{2/3}\gamma_1u\|_{H_h^{1/2}(\partial\mathcal{O})}
+h^{1/6}\|h^{2/3}k\gamma_0u\|_{H_h^{1/2}(\partial\mathcal{O})}\\
\leqslant &\; h^{5/6}\|\gamma_1u\|_{H_h^{1/2}(\partial\mathcal{O})}
+Ch^{5/6}\|\gamma_0u\|_{H_h^{3/2}(\partial\mathcal{O})}\\
\leqslant &\; Ch^{-2/3}\|u\|_{H_h^2(\mathbb{R}^n\setminus\mathcal{O})}.
\end{split}
\end{equation*}

Now we need to correct this operator with
\begin{equation*}
R_{+,w}:H^2(\mathbb{R}^n\setminus\mathcal{O})\to L^2(\partial\mathcal{O};\mathbb{C}^N),
\end{equation*}
and
\begin{equation*}
R_{-,w}:L^2(\partial\mathcal{O};\mathbb{C}^N)\to L^2(\mathbb{R}^n\setminus\mathcal{O}).
\end{equation*}
They are obtained by quantizing the symbols appeared in section \ref{sec:model}. Let $e^{\lambda,\delta}_{j,\mu}$ be as in \eqref{def:mueigens}, then we shall define
\begin{equation}
R_{+,w}=\Op_{\Sigma_w,h}(\tilde{e}_w^\delta)
:L^2(\mathbb{R}^n\setminus\partial\mathcal{O})\to L^2(\partial\mathcal{O};\mathbb{C}^N),
\end{equation}
where
\begin{equation*}
\tilde{e}_w^\delta\in S_{\Sigma_w,2/3}(\partial\mathcal{O};1,
\mathcal{L}(L^2[0,\infty);\mathbb{C}^N))
\end{equation*}
is given by
\begin{equation*}
\tilde{e}_w^\delta(j)u(p)=\int_0^\infty h^{-1/3}\chi(x_n)e_{j,\mu}^{\lambda,\delta}(h^{-2/3}x_n)u(x_n)dx_n, p\in T^\ast\partial\mathcal{O}
\end{equation*}
with $\lambda=h^{-2/3}(R(p)-w), \mu=Q(0,p)$. Similarly, the operator
$R_{-,w}$ can be defined as the formal adjoint of $R_{+,w}$ or more precisely,
\begin{equation*}
R_{-,w}=\Op_{\Sigma_w,h}((\tilde{e}_w^\delta)^\ast)
:L^2(\mathbb{R}^n\setminus\partial\mathcal{O})\to L^2(\partial\mathcal{O};\mathbb{C}^N),
\end{equation*}
where
\begin{equation*}
(\tilde{e}_w^\delta)^\ast\in S_{\Sigma_w,2/3}(\partial\mathcal{O};1,
\mathcal{L}(\mathbb{C}^N;L^2([0,\infty)))
\end{equation*}
is given by
\begin{equation*}
\tilde{e}_w^\delta u_-(p)=\sum_{j=1}^N
h^{-1/3}\chi(x_n)e_{j,\mu}^{\lambda,\delta}(h^{-2/3}x_n)u_-(j), p\in T^\ast\partial\mathcal{O}.
\end{equation*}

Then we have the Grushin problem for
\begin{equation}
\mathcal{P}_w(z)=
\left(
  \begin{array}{cc}
    P_w-z & R_{-,w} \\
    \gamma_1 & 0 \\
    R_{+,w} & 0 \\
  \end{array}
\right):\mathcal{B}_{w,r}\to\mathcal{H}_{w,r}.
\end{equation}
Our goal is to construct an inverse of $\mathcal{P}_w(z)$ for all $h$ small depending on $\delta$,
\begin{equation}
\mathcal{E}_w(z)=
\left(
  \begin{array}{ccc}
    E_w(z) & K_w(z) & E_{w,+}(z) \\
    E_{w,-}(z) & K_{w,-}(z) & E_{w,-+}(z) \\
  \end{array}
\right):\mathcal{H}_{w,r}\to\mathcal{B}_{w,r}
\end{equation}
where $E_{w,-+}(z)$ has nice properties that will be specified later.

\subsection{Construction of the inverse operator}
To construct the inverse operator, we first separate to three different parts: near the boundary and glancing hypersurface, near the boundary away from the glancing hypersurface and away from the boundary. In this section, we again work with $w=1$ for simplicity and it will be clear that the analysis is uniform for $w$ in a fixed compact subset of $(0,\infty)$.

We consider the case near the boundary and glancing hypersurface first. Let us translate the space $\mathcal{B}_{z,\lambda,r}$ and $\mathcal{H}_{z,\lambda,r}$ in section \ref{sec:model} into the $x_n$-coordinates and scale it by $h^{1/3}$ due to the change of coordinates. In this stage, we drop the dependence on $z$ and introduce the same weight function $\psi$ as previous.
\begin{equation*}
\begin{split}
\left\|\left(
         \begin{array}{c}
           u \\
           u_- \\
         \end{array}
       \right)
\right\|_{\mathcal{B}_{\lambda,r}}=&\;
h^{-2/3}\|e^{r\psi(x_n)/2h^{2/3}}
(hD_{x_n})^2u\|_{L^2([0,\infty))}+h^{-2/3}\|e^{r\psi(x_n)/2h^{2/3}}
x_nu\|_{L^2([0,\infty))}\\
&+\langle\lambda\rangle
\|e^{r\psi(x_n)/2h^{2/3}}u\|_{L^2(\mathbb{R}^n\setminus\mathcal{O})}
+h^{1/3}|u_-|_{\mathbb{C}^N},\\
\left\|\left(
         \begin{array}{c}
           v \\
           v_0 \\
           v_+
         \end{array}
       \right)
\right\|_{\mathcal{H}_{\lambda,r}}=&\;\|e^{r\psi(x_n)/2h^{2/3}}
v\|_{L^2([0,\infty))}
+h^{1/3}\langle\lambda\rangle^{1/4}|v_0|_{\mathbb{C}}
+h^{1/3}\langle\lambda\rangle|v_+|_{\mathbb{C}^N}.
\end{split}
\end{equation*}

\begin{lem}
Let $0<\epsilon<2/3$, $\chi_1\in\Psi^{0,0}(\partial\mathcal{O})$ be such that $\WF_h(\chi_1-\Id)\subset\{m:d(m,\Sigma)\geqslant C\}$ and $\WF_h(\chi_1)\subset\{m:d(m,\Sigma)\leqslant2C\}$. Then there exists $\mathcal{E}^L_1(z),\mathcal{E}^R_1(z)\in\Psi_{\Sigma,2/3}
(\partial\mathcal{O};
1,\mathcal{L}(\mathcal{H}_{\lambda,r},\mathcal{B}_{\lambda,r}))$ such that
\begin{equation*}
\mathcal{E}_1^L(z)\mathcal{P}(z)=
\chi_1\left(
  \begin{array}{cc}
    \chi(x_n/h^\epsilon) & 0 \\
    0 & \Id \\
  \end{array}
\right)+\mathcal{R}^L_1(z),
\end{equation*}
\begin{equation*}
\mathcal{P}(z)\mathcal{E}_1^R(z)=
\chi_1\left(
  \begin{array}{ccc}
    \chi(x_n/h^\epsilon) & 0 & 0 \\
    0 & \Id & 0 \\
    0 & 0 & \Id \\
  \end{array}
\right)+\mathcal{R}^R_1(z),
\end{equation*}
where the remainder terms satisfy
\begin{equation*}
\begin{split}
\mathcal{R}_1^L(z)\in
\Psi_{\Sigma,2/3}(\partial\mathcal{O};h^N\langle\lambda\rangle^{-N},
\mathcal{L}(\mathcal{B}_{\lambda,r},\mathcal{B}_{\lambda,r}))\\
\mathcal{R}_1^R(z)\in
\Psi_{\Sigma,2/3}(\partial\mathcal{O};h^N\langle\lambda\rangle^{-N},
\mathcal{L}(\mathcal{H}_{\lambda,r},\mathcal{H}_{\lambda,r}))
\end{split}
\end{equation*}
for any $N$.
\end{lem}
\begin{proof}
From the previous section, we can construct an operator $\tilde{\mathcal{E}}_1\in\Psi_{\Sigma,2/3}(\partial\mathcal{O};1,
\mathcal{L}(\mathcal{H}_{\lambda,r},\mathcal{B}_{\lambda,r}))$ with
$\WF_h(\tilde{\mathcal{E}}_1)\subset\{m:d(m,\Sigma)\leqslant2C\}$ such that
\begin{equation*}
\tilde{\mathcal{E}}_1(z)\mathcal{P}(z)=\Id+\tilde{R}_1^L(z),\;\;\;
\mathcal{P}(z)\tilde{\mathcal{E}}_1(z)=\Id+\tilde{R}_1^R(z).
\end{equation*}
Here the remainder term $\tilde{R}_1^L$ satisfies that for any $A\in\Psi^{0,0}(\partial\mathcal{O})$ with $\WF_h(A)\subset\{m:d(m,\Sigma)\leqslant C\}$ and any $k$,
\begin{equation*}
A\tilde{R}_1^L=
\left(
  \begin{array}{cc}
    x_n^k & 0  \\
    0 & 0 \\
  \end{array}
\right)B_k^L+h^kA_k^L,
\end{equation*}
with
\begin{equation*}
A_k^L,B_k^L\in\Psi_{\Sigma,2/3}(\partial\mathcal{O};1,
\mathcal{L}(\mathcal{B}_{\lambda,r},\mathcal{B}_{\lambda,r})).
\end{equation*}
We notice that for $0<\epsilon<2/3$, the operator
\begin{equation*}
\left(
  \begin{array}{cc}
    \chi(x_n/h^\epsilon) & 0 \\
    0 & \Id \\
  \end{array}
\right)
\end{equation*}
is bounded on $\mathcal{B}_{\lambda,r}$. In fact, in $t$ coordinates, this becomes $\chi(h^{2/3-\epsilon}t)$ whose derivatives are all bounded. Therefore we can set
\begin{equation*}
\mathcal{E}_l^L(z)=\chi_1
\left(
  \begin{array}{cc}
    \chi(x_n/h^\epsilon) & 0 \\
    0 & \Id \\
  \end{array}
\right)\tilde{\mathcal{E}}_1(z).
\end{equation*}
Since $\langle\lambda\rangle=O(h^{-2/3})$, it is clear that this operator satisfies the condition. Similarly, we can construct
\begin{equation*}
\mathcal{E}_l^R(z)=\tilde{\mathcal{E}}_1(z)\chi_1
\left(
  \begin{array}{ccc}
    \chi(x_n/h^\epsilon) & 0 & 0\\
    0 & \Id & 0\\
    0 & 0 & \Id \\
  \end{array}
\right).
\end{equation*}
\end{proof}

Now for the case near the boundary but away from the glancing hypersurface, the spaces $\mathcal{H}_{\lambda,r}^\#$ and $\mathcal{B}_{\lambda,r}^\#$ becomes
\begin{equation*}
\begin{split}
\left\|\left(
         \begin{array}{c}
           u \\
           u_- \\
         \end{array}
       \right)
\right\|_{\mathcal{B}_{\lambda,r}^\#}=&\;
h^{-2/3}\|e^{r\psi(x_n)/2h^{2/3}}
(hD_{x_n})^2u\|_{L^2([0,\infty))}+\langle\lambda\rangle
\|e^{r\psi(x_n)/2h^{2/3}}u\|_{L^2(\mathbb{R}^n\setminus\mathcal{O})}
+h^{1/3}|u_-|_{\mathbb{C}^N},\\
\left\|\left(
         \begin{array}{c}
           v \\
           v_0 \\
           v_+
         \end{array}
       \right)
\right\|_{\mathcal{H}_{\lambda,r}}=&\;\|e^{r\psi(x_n)/2h^{2/3}}
v\|_{L^2([0,\infty))}
+h^{1/3}\langle\lambda\rangle^{1/4}|v_0|_{\mathbb{C}}
+h^{1/3}\langle\lambda\rangle|v_+|_{\mathbb{C}^N},
\end{split}
\end{equation*}
in the $x_n$-coordinates. In this situation, we have

\begin{lem}
Let $0<\epsilon<2/3$, $\chi_2\in\Psi^{0,0}(\partial\mathcal{O})$ be such that $\WF_h(\chi_2-\Id)\subset\{m:d(m,\Sigma)\leqslant C\}$ and $\WF_h(\chi_2)\subset\{m:d(m,\Sigma)\geqslant\frac{1}{2}C\}$. Then there exists $\mathcal{E}^L_2(z),\mathcal{E}^R_2(z)\in\Psi_{\Sigma,2/3}
(\partial\mathcal{O};
1,\mathcal{L}(\mathcal{H}_{\lambda,r},\mathcal{B}_{\lambda,r}))$ such that
\begin{equation*}
\mathcal{E}_2^L(z)\mathcal{P}(z)=
\chi_2\left(
  \begin{array}{cc}
    \chi(x_n/h^\epsilon) & 0 \\
    0 & \Id \\
  \end{array}
\right)+\mathcal{R}^L_2(z),
\end{equation*}
\begin{equation*}
\mathcal{P}(z)\mathcal{E}_2^R(z)=
\chi_2\left(
  \begin{array}{ccc}
    \chi(x_n/h^\epsilon) & 0 & 0 \\
    0 & \Id & 0 \\
    0 & 0 & \Id \\
  \end{array}
\right)+\mathcal{R}^R_2(z),
\end{equation*}
where the remainder terms satisfy
\begin{equation*}
\begin{split}
\mathcal{R}_2^L(z)\in
\Psi_{\Sigma,2/3}(\partial\mathcal{O};h^N\langle\lambda\rangle^{-N},
\mathcal{L}(\mathcal{B}_{\lambda,r}^\#,\mathcal{B}_{\lambda,r}^\#))\\
\mathcal{R}_2^R(z)\in
\Psi_{\Sigma,2/3}(\partial\mathcal{O};h^N\langle\lambda\rangle^{-N},
\mathcal{L}(\mathcal{H}_{\lambda,r}^\#,\mathcal{H}_{\lambda,r}^\#))
\end{split}
\end{equation*}
for any $N$.
\end{lem}
\begin{proof}
We can repeat the same argument with the standard semiclassical calculus and notice that $\langle\lambda\rangle=O(h^{-2/3}\langle\xi'\rangle^2)$ to get the properties of the remainder.
\end{proof}

Now combining the two lemmas above, we get the approximated inverse near the boundary. More precisely,
\begin{prop}
There exists $\mathcal{E}^L(z),\mathcal{E}^R(z)
:\mathcal{H}_r\to\mathcal{B}_{r,\epsilon}$ such that
\begin{equation*}
\mathcal{E}^L(z)\mathcal{P}(z)=
\left(
  \begin{array}{cc}
    \chi(x_n/h^\epsilon) & 0 \\
    0 & \Id \\
  \end{array}
\right)+\mathcal{R}^L(z),
\end{equation*}
\begin{equation*}
\mathcal{P}(z)\mathcal{E}^R(z)=
\left(
  \begin{array}{ccc}
    \chi(x_n/h^\epsilon) & 0 & 0 \\
    0 & \Id & 0 \\
    0 & 0 & \Id \\
  \end{array}
\right)+\mathcal{R}^R(z),
\end{equation*}
where the remainder terms satisfy
\begin{equation*}
\begin{split}
\langle h^2\Delta_{\partial\mathcal{O}}\rangle^N\mathcal{R}_3^L(z)
\langle h^2\Delta_{\partial\mathcal{O}}\rangle^N=O(h^N)
:\mathcal{B}_{r,\epsilon}\to\mathcal{B}_{r,\epsilon}\\
\langle h^2\Delta_{\partial\mathcal{O}}\rangle^N\mathcal{R}_3^R(z)
\langle h^2\Delta_{\partial\mathcal{O}}\rangle^N=O(h^N)
:\mathcal{H}_r\to\mathcal{H}_r,\\
\end{split}
\end{equation*}
for any $N$. Here $\langle h^2\Delta_{\partial\mathcal{O}}\rangle^N$
applies to all the components and the spaces $\mathcal{B}_{r,\epsilon}$ are defined as $\mathcal{B}_{r,\delta}$ further truncated to the $h^\epsilon$-neighborhood of the boundary by $\chi(x_n/h^{\epsilon})$.. Moreover, the $-+$-components for the approximate inverses
satisfy
\begin{equation*}
E_{-+}^L(z)\equiv E_{-+}^R(z)\in
\Psi_{\Sigma,2/3}^{0,1,2}(\partial\mathcal{O};
\mathcal{L}(\mathbb{C}^N,\mathbb{C}^N)).
\end{equation*}
\end{prop}
\begin{proof}
We can simply choose $\chi_1$ and $\chi_2$ such that $\chi_1+\chi_2=1$ and set $\mathcal{E}^\cdot(z)=\mathcal{E}^\cdot_1(z)+\mathcal{E}^\cdot_2(z)$,
$\cdot=L,R$. To prove the last statement, we notice that from the construction,
\begin{equation*}
E_{-+}^L=\chi_1\tilde{E}_{-+1}+\chi_2\tilde{E}_{-+2},\;\;\;
E_{-+}^R=\tilde{E}_{-+1}\chi_1+\tilde{E}_{-+2}\chi_2.
\end{equation*}
Near the glancing hypersurface, $\{m:d(m,\Sigma)\leqslant\frac{1}{2}C\}$, $\chi_1\equiv\Id$ while $\chi_2\equiv0$. Away from the glancing hypersurface
$\{m:d(m,\Sigma)\geqslant2C\}$, $\chi_1\equiv0$ while $\chi_2\equiv\Id$. In the intermediate region, $E_{-+1}\equiv E_{-+2}$ from our discussion in section \ref{sec:inter}. Therefore $E_{-+}^L$ and $E_{-+}^R$ are essentially the same in the $\Psi_{\Sigma,2/3}^{0,1,2}(\partial\mathcal{O};
\mathcal{L}(\mathbb{C}^N,\mathbb{C}^N))$.
\end{proof}

Finally, we can combine this with the estimate away from the boundary to get the inverse.
\begin{prop}
Let $0<\epsilon<2/3$, $0<h<h_0(\delta)$, there exists $\mathcal{E}_w(z):\mathcal{H}_{w,0}\to\mathcal{B}_{w,0,\epsilon}$ such that
\begin{equation*}
\mathcal{P}_w(z)\mathcal{E}_w(z)=\Id,\;\;\;
\mathcal{E}_w(z)\mathcal{P}_w(z)=\Id
\end{equation*}
and $E_{w,-+}\in\Psi_{\Sigma_w,2/3}^{0,1,2}(\partial\mathcal{O};
\mathcal{L}(\mathbb{C}^N,\mathbb{C}^N))$.
\end{prop}
\begin{proof}
Let us begin with an approximate right inverse
\begin{equation*}
\mathcal{\mathcal{E}}^R(z)=\mathcal{E}^R(z)
\left(
  \begin{array}{ccc}
    \tilde{\chi}(x_n/h^\epsilon) & 0 & 0\\
    0 & \Id & 0\\
    0 & 0 & \Id \\
  \end{array}
\right)
+\left(
  \begin{array}{ccc}
    E_\epsilon(1-\tilde{\chi}(x_n/h^\epsilon)) & 0 & 0\\
    0 & 0 & 0 \\
  \end{array}
\right).
\end{equation*}
Here $\tilde{\chi}\in C^\infty([0,\infty))$ supported in $\{\chi=1\}$. Then we can compute
\begin{equation*}
\mathcal{P}(z)\tilde{\mathcal{E}}^R(z)=\Id+\mathcal{K}^R(z)
\end{equation*}
where the remainder is given by
\begin{equation*}
\mathcal{K}^R(z)=\mathcal{R}^R(z)
\left(
  \begin{array}{ccc}
    \tilde{\chi}(x_n/h^\epsilon) & 0 & 0 \\
    0 & \Id & 0 \\
    0 & 0 & \Id \\
  \end{array}
\right)+
\left(
  \begin{array}{ccc}
    K_\epsilon(1-\tilde{\chi}(x_n/h^\epsilon)) & 0 & 0 \\
    \gamma E_\epsilon(1-\tilde{\chi}(x_n/h^\epsilon)) & 0 & 0 \\
    R_+E_\epsilon(1-\tilde{\chi}(x_n/h^\epsilon)) & 0 & 0 \\
  \end{array}
\right).
\end{equation*}
Since $E_\epsilon(1-\tilde{\chi})u$ is supported away from the boundary, we have $\gamma E_\epsilon(1-\tilde{\chi}(x_n/h^\epsilon))=0$. Moreover,
for any smooth $u$, since $(1-\tilde{\chi}(x_n/h^{\epsilon}))u$ is supported in $D(h^\epsilon)$, $E_\epsilon(1-\tilde{\chi}(x_n/h^{\epsilon}))u$ is supported in $(D(1-\gamma)h^{\epsilon})$, so by the super-exponential decay of $e_{j,\mu}^{\lambda,\delta}$, we have
\begin{equation}
\tilde{e}_w^\delta(j)u(p,x_n)=\int_0^\infty h^{-1/3}\chi(x_n)e_{j,\mu}^{\lambda,\delta}(h^{-2/3}x_n)u(p,x_n)dx_n
=O(h^\infty)
\end{equation}
which gives $R_+E_\epsilon(1-\tilde{\chi}(x_n/h^\epsilon))=O(h^\infty)$.
Therefore we get $\mathcal{K}^R=O(h^\infty):\mathcal{H}_0\to\mathcal{H}_0$ and hence for $h$ small enough, $(\Id+\mathcal{K}^R)^{-1}=\Id+\mathcal{A}$ where $\mathcal{A}=O(h^\infty):\mathcal{H}_0\to\mathcal{H}_0$. We can now put
\begin{equation*}
\mathcal{E}(z)=\mathcal{E}^R(z)(\Id+\mathcal{A}(z))
\end{equation*}
Suppose
\begin{equation*}
\mathcal{A}(z)=
\left(
  \begin{array}{ccc}
    A_{11}(z) & A_{12}(z) & A_{13}(z) \\
    A_{21}(z) & A_{22}(z) & A_{23}(z) \\
    A_{31}(z) & A_{32}(z) & A_{33}(z) \\
  \end{array}
\right)
\end{equation*}
then from the formula of $\mathcal{K}^R$, we see it is lower triangular and thus the same is true for $\mathcal{A}$. Therefore
\begin{equation*}
E_{-+}(z)=E_{-+}^R(z)+E_{-+}^R(z)A_{33}(z)
\end{equation*}
Here $A_{33}(z)\in\Psi^{-\infty,-\infty}(\partial\mathcal{O};
\mathcal{L}(\mathbb{C}^N,\mathbb{C}^N))$ since it comes entirely from $\mathcal{R}_3^R$. Therefore $E_{-+}(z)\in\Psi_{\Sigma_w,2/3}^{0,1,2}(\partial\mathcal{O};
\mathcal{L}(\mathbb{C}^N,\mathbb{C}^N))$ is essentially the same as $E_{-+}^R$ (and also as $E_{-+}^L$).
\end{proof}

\subsection{Reduction to $E_{-+}$}
Now we state the main result of this section.
\begin{thm}
\label{thm:e+-}
Assume that $W$ is a fixed compact subset of $(0,\infty)$ and $\epsilon\ll1$. For every $w\in W$ and $z\in\mathbb{C}$ such that $|\Re z|\ll1/\delta$, $|\Im z|\leqslant C_1$, there exists
\begin{equation}
E_{w,-+}(z)\in\Psi_{\Sigma_w,2/3}^{0,1,2}
\end{equation}
where
$\Sigma_w=\{p\in T^\ast\partial\mathcal{O}:R(p)=w\}$, $N=N(C_1)$ such that for $0<h<h_0$ and some large $C>0$:

(i) The multiplicity of resonances are given by
\begin{equation}
\label{eq:mult}
m_\mathcal{O}(h^{-2}(w+h^{2/3}z))=\frac{1}{2\pi i}
\tr\oint_{|\tilde{z}-z|=\epsilon}E_{w,-+}(\tilde{z})^{-1}
\frac{d}{d\tilde{z}}E_{w,-+}(\tilde{z})d\tilde{z}
\end{equation}

(ii) If $E_{w,-+}^0(z;p,h)=\sigma_{\Sigma,h}(E_{w,-+}(z))(p;h)$, $p\in T^\ast\partial\mathcal{O}$, then
\begin{equation}
E_{w,-+}^0(z,p,h)=O(\langle\lambda-\Re z\rangle):\mathbb{C}^N\to\mathbb{C}^N,
\end{equation}
where $\lambda=h^{-2/3}(R(p)-w)$.

(iii) For $|\lambda|\leqslant1/C\sqrt{\delta}$,
\begin{equation}
\|E_{w,-+}^0(z;p,h)-\diag(z-\lambda-e^{-2\pi i/3}
\zeta_j'(p))\|_{\mathcal{L}(\mathbb{C}^N,\mathbb{C}^N)}
\leqslant\epsilon.
\end{equation}
Moreover, $\det E_{w,-+}^0(z;p,h)=0$ if and only if
\begin{equation}
z=\lambda+e^{-2\pi i/3}\zeta_j'(p)
\end{equation}
for some $1\leqslant j\leqslant N$ and all zeroes are simple. Here $\zeta_j'(p)=\zeta_j'(2Q(p))^{2/3}$.

(iv) For $|\lambda|\geqslant1/C\sqrt{\delta}$, $E_{w,-+}^0$ is invertible and
\begin{equation}
E_{w,-+}^0(z,p,h)^{-1}=O(\langle\lambda-\Re z\rangle^{-1}):\mathbb{C}^N\to\mathbb{C}^N,
\end{equation}
\end{thm}
\begin{proof}
The statement (i) follows from the formula
\begin{equation*}
\left(
  \begin{array}{c}
    h^{-2/3}(P(h)-w)-z \\
    \gamma \\
  \end{array}
\right)^{-1}=(E_w(z),K_w(z))
-E_{w,+}(z)E_{w,-+}(z)^{-1}(E_{w,-}(z),K_{w,-}(z)).
\end{equation*}
The other statements follow directly from our construction of $\mathcal{E}_w$.
\end{proof}

\section{Proof of the theorem}
\label{sec:resfree}

\subsection{Resonance Bands}
We first prove Theorem \ref{thm:main1}. Under the pinched curvature condition, we have
\begin{equation*}
K\zeta_j'<\kappa\zeta_{j+1}',\;\;\; 1\leqslant j\leqslant j_0
\end{equation*}
which can be translated to
\begin{equation*}
\max_{p\in\Sigma}\zeta_j'(p)<\min_{p\in\Sigma}\zeta_{j+1}'(p),\;\;\; 1\leqslant j\leqslant j_0.
\end{equation*}

Suppose $\lambda$ is a resonance which satisfies that for some $1\leqslant j\leqslant j_0$,
\begin{equation*}
K\zeta_j'(\Re\lambda)^{1/3}+C\leqslant-\Im\lambda
\leqslant\kappa\zeta_{j+1}'(\Re\lambda)^{1/3}-C.
\end{equation*}

Let $\zeta=\lambda^2=h^{-2}(1+h^{2/3}z)$ and $h=(\Re\lambda)^{-1}$, then we have
\begin{equation*}
K\zeta_j'h^{1/3}+C\leqslant-\Im\lambda\leqslant\kappa\zeta_{j+1}'h^{1/3}-C
\end{equation*}
and
\begin{equation*}
\Re z=h^{-2/3}(h^2\Re\zeta-1)=O(h^{2/3}).
\end{equation*}
\begin{equation*}
-\Im z=h^{-2/3}(-h^2\Im\zeta)=-2h^{1/3}\Im\lambda\in
[2K\zeta_j'+Ch^{1/3},2\kappa\zeta_{j+1}'-Ch^{1/3}].
\end{equation*}
Therefore for $p\in\Sigma_1$, i.e. $R(p)=1$,
\begin{equation*}
\Im[z-\lambda-e^{-2\pi i/3}\zeta_k'(p)]
=\Im z+\zeta_k'(2Q(p))^{2/3}\cos(\pi/6)
\in[\Im z+2\kappa\zeta_k',\Im z+2K\zeta_k']
\end{equation*}
thus for at most one of $k\in\{j,j+1\}$,
\begin{equation*}
|\Im[z-\lambda-e^{-2\pi i/3}\zeta_k'(p)]|\geqslant Ch^{1/3}
\end{equation*}
while for all other $k\in\{1,\ldots,j_0\}$,
\begin{equation*}
|\Im[z-\lambda-e^{-2\pi i/3}\zeta_k'(p)]|\geqslant \frac{1}{O(1)}.
\end{equation*}

Therefore we can decompose
\begin{equation*}
E_{-+}(z):=E_{1,-+}(z)=A(z)G_{-+}(z)B(z)
\end{equation*}
where
\begin{equation*}
A(z),B(z)\in\Psi_{\Sigma_1,2/3}^{0,0,0}(\partial\mathcal{O};
\mathcal{L}(\mathbb{C}^N,\mathbb{C}^N))
\end{equation*}
are invertible and
\begin{equation*}
G_{-+}(z)\in\Psi^{0,1,2}_{\Sigma_1,2/3}(\partial\mathcal{O};
\mathcal{L}(\mathbb{C}^N,\mathbb{C}^N))
\end{equation*}
has principal symbol $G_{-+}^0(z)$, such that, near $\Sigma_1$,
\begin{equation*}
\Im G_{-+}^0(z)\geqslant C_0h^{1/3}\Id_{\mathbb{C}^N}
\end{equation*}
while away from $\Sigma_1$,
\begin{equation*}
\Im G_{-+}^0(z)\geqslant \frac{1}{O(1)}h^{-2/3}\langle\xi\rangle^2.
\end{equation*}

Now we choose $C_0$ large enough, then we see that the imaginary part of the total symbol of $G_{-+}(z)$ is bounded below by a positive symbol in $S_{\Sigma_1,2/3}^{-1/3,0,2}$. The sharp G\aa rding's inequality gives
\begin{equation*}
\|E_{-+}(z)u\|_{L^2}\geqslant C\|G_{-+}(z)u\|_{L^2}\geqslant Ch^{1/3}\|u\|_{L^2},\;\;\; \forall u\in C^\infty(\partial\mathcal{O};\mathbb{C}^N).
\end{equation*}

Therefore $E_{-+}(z)$ is invertible for $0<h\leqslant h_0$. Therefore when $\Re\lambda\geqslant C=h_0^{-1}$, it cannot be a resonance.

\subsection{Weyl's Law}
In this part, we sketch the proof of Theorem \ref{thm:weyl}. See \cite[Section 9-10]{SZ6} for details of the proof.

Heuristically, we want to use the symbol of $E_{w,-+}(z)$ to compute its trace, then use \eqref{eq:mult} to count the number of resonances. However, this operator is not in the trace class. The first step is to construct a finite-rank approximation $\tilde{E}_{w,-+}(z)\in\Psi_{\Sigma_w,2/3}^{0,1,2}
(\partial\mathcal{O};\mathcal{L}(\mathbb{C}^N,\mathbb{C}^N))$ which is invertible and such that
\begin{equation*}
\tilde{E}_{w,-+}(z)^{-1},\;\;(\Lambda_w^{-1}\tilde{E}_{w,-+}(z))^{-1},
\;\;\tilde{E}_{w,-+}(z)^{-1}E_{w,-+}(z)=O(1):
L^2(\partial\mathcal{O};\mathbb{C}^N)\to L^2(\partial\mathcal{O};\mathbb{C}^N)
\end{equation*}
where $\Lambda_w=\langle h^{-2/3}(-h^2\Delta_{\partial\mathcal{O}}-w)\rangle
\in\Psi_{\Sigma_w,2/3}^{0,1,2}$ is elliptic. Moreover, we have $E_{w,-+}(z)-\tilde{E}_{w,-+}(z)$ is independent of $z$ and of rank $M=O(Lh^{1-n+2/3})$. Microlocally $\tilde{E}$ is only different from $E$ on the the glancing region where $E$ is not invertible.

From this finite-rank approximation, we can solve another Grushin problem to reduce $E_{w,-+}$ to a finite matrix. More precisely, we consider
\begin{equation}
\label{2ndgru}
\mathcal{Q}_w(z)=\left(
                 \begin{array}{cc}
                   \Lambda^{-1}E_{w,-+}(z) & R_{w,-}(z) \\
                   R_{w,+}(z) & 0 \\
                 \end{array}
               \right): L^2(\partial\mathcal{O};\mathbb{C}^N)\times\mathbb{C}^M
               \to L^2(\partial\mathcal{O};\mathbb{C}^N)\times\mathbb{C}^M,
\end{equation}
with bounded inverse
\begin{equation*}
\mathcal{F}_w(z)=\left(
                 \begin{array}{cc}
                   F_w(z)\Lambda & F_{w,+}(z) \\
                   F_{w,-}(z) & F_{w,-+}(z) \\
                 \end{array}
               \right): L^2(\partial\mathcal{O};\mathbb{C}^N)\times\mathbb{C}^M
               \to L^2(\partial\mathcal{O};\mathbb{C}^N)\times\mathbb{C}^M.
\end{equation*}
The construction of the Grushin problem is as follows: Let $e_1,\ldots,e_M$ be an orthonormal basis of the image of $\Lambda_w^{-1}(E_{w,-+}(z)-\tilde{E}_{w,-+}(z))^\ast$, then we set
\begin{equation*}
R_{w,+}u(j)=\langle u,e_j\rangle,\;\;\;1\leqslant j\leqslant M;\;\;\;
R_{w,-}(z)u_-=\Lambda^{-1}\tilde{E}_{w,-+}(z)R_{w,+}^\ast u_-.
\end{equation*}
The inverse is given by
\begin{equation*}
\begin{split}
F_w(z)=&\;(I-R_{w,+}^\ast R_{w,+})\tilde{E}_{w,-+}(z)^{-1},\\
F_{w,+}(z)=&\;R_{w,+}^\ast-(I-R_{w,+}^\ast R_{w,+})\tilde{E}_{w,-+}(z)^{-1}E_{w,-+}(z)R_{w,+}^\ast,\\
F_{w,-}(z)=&\;R_{w,+}\tilde{E}_{w,-+}(z)^{-1},\\
F_{w,-+}(z)=&\;-R_{w,+}\tilde{E}_{w,-+}(z)^{-1}E_{w,-+}(z)R_{w,+}^\ast.
\end{split}
\end{equation*}

With these preparation, we can prove a local trace formula on the scale 1 in the $z$ variable for every $w$. This is on the scale $h^{2/3}$ for the semiclassical variable $w+h^{2/3}z$ which is the square of the resonances $h^2\lambda^2$. We remark that this is the largest scale that we can work with for each fixed $w$ since the whole microlocal framework is built exactly on such scale.

For the $j_0$-th band of the resonances, we consider a domain
\begin{equation*}
W=\left\{-\frac{1}{2}L<\Re z<\frac{1}{2}L, A_-<-\Im z<A_+\right\}
\end{equation*}
where
\begin{equation*}
2K\zeta_{j_0-1}'<A_-<2\kappa\zeta_{j_0}'\leqslant
2K\zeta_{j_0}'<A_+<2\kappa\zeta_{j_0+1}'.
\end{equation*}
Let $\partial W=\gamma=\gamma_1\cup\gamma_2\cup\gamma_3\cup\gamma_4$ be the boundary of $W$, where $\gamma_1$ and $\gamma_3$ are the horizontal segments while $\gamma_2$ and $\gamma_4$ are the vertical segments. If we write $\Res_w(h)=\{z:m_\mathcal{O}(h^{-2}(w+h^{2/3}z))>0\}$, then we have the local trace formula
\begin{equation}
\label{eq:localtr}
\begin{split}
\sum_{z\in\Res_w(h)\cap W}f(z)=&\;\sum_{j=1,3}\tr\frac{1}{2\pi i}
\int_{\gamma_j}f(z)\left[E_{w,-+}(z)^{-1}\frac{d}{dz}E_{w,-+}(z)\right.\\
&\;\;\;\left.-\tilde{E}_{w,-+}(z)^{-1}
\frac{d}{dz}\tilde{E}_{w,-+}(z)\right]dz
+O(Lh^{1-n+2/3})\\
\end{split}
\end{equation}
for any holomorphic function $f$ defined near $W$ such that $|f(z)|\leqslant1$ near $\gamma_2\cup\gamma_4$. (In fact, to make this argument work, we need to choose a slightly larger rectangular contour around $W$ and $f$ holomorphic in an even larger domain. Also we need to the contour does not pass through any of the poles of $E_{w,-+}^{-1}$. These technical issues are handled in \cite{SZ6}.)

The main idea to prove this local trace formula is to change the trace of $E_{-+}^{-1}E_{-+}'-\tilde{E}_{-+}^{-1}\tilde{E}_{-+}'$ to the trace of $F_{-+}^{-1}F_{-+}'=\log\det F_{-+}$ by using the Grushin problem \eqref{2ndgru} constructed above. We observe that $F_{-+}$ is an $M\times M$ matrix which is $O(1):\mathbb{C}^M\to\mathbb{C}^M$ under the standard norm. This shows that $\log\det F_{-+}=O(M)=O(Lh^{1-n+2/3})$ and thus all the contributions from the two vertical segment can be controlled by $O(Lh^{1-n+2/3})$ using lower modulus theorem. Notice that this characterization of resonances by the poles of $F_{-+}^{-1}$ also gives a local upper bound on the number of the resonances
\begin{equation}
\label{eq:locupper}
\sum_{|\Re\zeta-1|\leqslant Ch^{2/3},0<-\Im\zeta<Ch^{2/3}}m_\mathcal{O}(\zeta)=O(h^{1-n+2/3}).
\end{equation}

In the local trace formula \eqref{eq:localtr}, we use the (second microlocalization) symbol to compute the trace on the right-hand side and get
\begin{equation}
\label{eq:loccount}
\begin{split}
\sum_{z\in\Res_w(h)\cap W}f(z)=&\frac{h^{1-n+2/3}}{(2\pi)^{n-1}}
\int_{\Sigma_w\times\mathbb{R}}f(\lambda+e^{-2\pi i/3}\zeta_{j_0}'(q))
1_{I(q)}(s)L_{\Sigma_w}(dq)ds\\
&\;\;\;+O(Lh^{1-n+2/3})+O_{f,L}(h^{2-n})\\
\end{split}
\end{equation}
where $(q,s)\in\Sigma_w\times\mathbb{R}$ is a local coordinates for a neighborhood of $\Sigma_w\in T^\ast\partial\mathcal{O}$ such that $s|_{\Sigma_w}=0$, $L_{\Sigma_w}(dq)ds$ is the Liouville measure on $T^\ast X$, and
\begin{equation*}
I(q)=\{s\in\mathbb{R}:s+e^{-2\pi i/3}\zeta_{j_0}'(q)\in W\}.
\end{equation*}

For fixed $L$ (and say $f=1$), this does not give a better description of resonances than the upper bound \eqref{eq:locupper}. However, if we make $L$ large (which does not change the principal symbol in our construction, but may potentially affect the lower order terms), and choose $f$ suitably, we can get a better estimate than \eqref{eq:locupper}. The idea is to let $f$ to be very large in $W$ away from the $\gamma_2\cup\gamma_4$ but remain bounded ($|f|\leqslant1$ as required from the assumption in \eqref{eq:localtr}) near $\gamma_2\cup\gamma_4$. A standard choice is the Gaussian functions
\begin{equation*}
f_\epsilon(z)=((1+O(\epsilon L))e^{-\epsilon L^2/2})^{-1}e^{-\epsilon
(z-z_0)^2},\;\; z_0=-\frac{1}{2}i(A_-+A_+),\;\; \epsilon L\ll1,\; \epsilon L^2\gg\log\frac{1}{\epsilon}.
\end{equation*}

Then from \eqref{eq:loccount} we obtain
\begin{equation*}
\sum_{z\in\Res_w(h)\cap W}\sqrt{\frac{\epsilon}{2\pi}}e^{-\epsilon(\Re(z-z_0))^2/2}
=(1+O(\epsilon L))\frac{h^{1-n+2/3}}{(2\pi)^{n-1}}\int_{\Sigma_w}L_{\Sigma_w}(dq)+
O_{\epsilon,L}(h^{2-n}).
\end{equation*}

Finally, we let $L=\epsilon^{-2/3}$ and integrate in $w$ to get the Weyl's law in the semiclassical setting
\begin{prop}(see \cite[Proposition 10.1]{SZ6})
For $0<a<b$, let
\begin{equation*}
N_h([a,b];j)=\sum_{a<\Re z<b,2\kappa\zeta_j'h^{2/3}<-\Im z<2K\zeta_j'h^{2/3}}m_\mathcal{O}(h^{-2}z).
\end{equation*}
Then under the assumption of \ref{thm:main1}, we have
\begin{equation}
\label{semi}
N_h([a,b];j)=(1+O(\epsilon))\frac{h^{1-n}}{(2\pi)^{n-1}}
\int_{a\leqslant|\xi'|_{x'}^2\leqslant b}dx'd\xi'+O_\epsilon(h^{1-n+1/3})
\end{equation}
for any $1\leqslant j\leqslant j_0$ and $\epsilon>0$.
\end{prop}

Now the Weyl law \eqref{weyl} follows from a dyadic decomposition of the interval $|\lambda|\leqslant r$ and applying \eqref{semi} for each dyadic piece of the interval.

\end{document}